\documentclass[11pt,a4paper,draft,reqno]{amsart} 
\usepackage{amssymb,amscd,amsmath, young,mathrsfs,mathabx}
\usepackage{mathrsfs, mathdots,rotating,multirow} 
\usepackage[mathcal]{eucal} 
\usepackage{float}
\usepackage[usenames]{color}
\usepackage{soul}
\usepackage{longtable}
\usepackage{graphicx}
\usepackage{kotex}

\usepackage[colorinlistoftodos]{todonotes}

\newcommand{\GL}{\mathrm{GL}}

\newcommand{\Q}{\mathbb{Q}}

\newcommand{\Z}{\mathbb{Z}}
\newcommand{\N}{\mathbb{N}}

\newcommand{\Mat}{\mathrm{Mat}}

\theoremstyle{plain}
\newtheorem{thm}{Theorem}[section]
\newtheorem{lem}[thm]{Lemma}
\newtheorem{pro}[thm]{Proposition}
\newtheorem{cor}[thm]{Corollary}
\newtheorem*{claim*}{Claim}

\newtheorem{rem}[thm]{Remark}

\theoremstyle{remark}

\newtheorem{qun}[thm]{Question}
\newtheorem{exm}[thm]{Example}

\newtheorem{dfn}[thm]{Definition}

\numberwithin{equation}{section}

\numberwithin{table}{section}
\newcommand{\bsm}{\boldsymbol{m}}

\newcommand{\tensor}{\otimes}

\newcommand{\mfp}{\mathfrak{p}}

\newcommand{\Gri}{\ensuremath{\mathcal{O}}}

\renewcommand{\epsilon}{\varepsilon}

\renewcommand{\phi}{\varphi}
\renewcommand{\theta}{\vartheta}

\newcommand{\mcO}{\mathcal{O}}

\newcommand{\Qp}{\mathbb{Q}_{p}}
\newcommand{\Op}{\mcO_{\mfp}}

\newcommand{\tud}{\textup{d}}

\newcommand{\ideal}{\triangleleft}

\newcommand{\bse}{\boldsymbol{e}}

\newcommand{\Zp}{\mathbb{Z}_{p}}

\def \bff {{\bf f}}

\def \bfg {{\bf g}}

\DeclareMathOperator{\Spec}{Spec}

\def \diag {\textup{diag}}
\def \co {\textup{co}}

\def \bff {{\bf f}}

\def \bfg {{\bf g}}

\def \bfx {{\bf x}}
\def \bfX {{\bf X}}
\def \bfy {{\bf y}}
\def \bfY {{\bf Y}}

\def \Fp {\ensuremath{\mathbb{F}_p}}

\def \mcR {\ensuremath{\mathcal{R}}}
\def \mcN {\ensuremath{\mathcal{N}}}

\def \p {\ensuremath{\mathfrak{p}}}
\def \Fq {\ensuremath{\mathbb{F}_q}}
\def \Zp  {\mathbb{Z}_p}

\def \Mat {\mathrm{Mat}}

\def \bsr {\boldsymbol{r}}

\def \bss {\boldsymbol{s}}

\author{Seok Hyeong Lee}\address{Center for Quantum Structures in Modules and Spaces, Seoul National University, Seoul 08826, South Korea}\thanks{The first author was supported by the National Research Foundation of Korea(NRF) grants No. 2020R1A5A1016126 and No. RS--2024-00462910.}\email{lshyeong@snu.ac.kr}

\author{Seungjai Lee}\thanks{The second author was supported by Incheon National University Research Grant in 2024.} \address{Department of Mathematics, Incheon National University, Incheon 22012, South Korea}\email{seungjai.lee@inu.ac.kr}

\begin{document}
\title{Cotype zeta functions enumerating subalgebras of $R$-algebras}
\begin{abstract}
We introduce and study subalgebra cotype zeta functions, multivariate zeta functions enumerating fixed-index subalgebras of $R$-algebras of a given cotype. This generalizes and unifies previous works on subalgebra zeta functions and cotype zeta functions of $R$-algebras. We prove the local functional equations for the generic Euler factors of these zeta functions, and give an explicit formula for the subalgebra cotype zeta function of a general $\Z$-Lie algebra $L$ of rank 3. We also give an asymptotic formula for the number of subalgebras $\Lambda$ of $L$ of index at most $X$ for which $L/\Lambda$ has rank at most, answering a question of Chinta, Kaplan, and Koplewitz. In particular, we show that unlike $\Z^3$, $\Z$-Lie algebras of rank 3 with additional multiplication structure exhibit different distribution of cocyclic subalgebras. 
\end{abstract}

	
\allowdisplaybreaks
	\maketitle
\section{Introduction}\label{sec:intro}
\subsection{Background and motivations} 
 Let $R$ be a Dedekind domain, finitely generated as a ring. Let $L$ be an $R$-algebra, free and finitely generated as an $R$-module endowed with a multiplication $L\otimes_{R}L\rightarrow L$ -- not necessarily associative, commutative, or unital. A \textit{subalgebra} of $L$ is an $R$-submodule $\Lambda$ that is stable under the given multiplication. For $m\in\N$, let $a_{m}(L)$ denote the number of subalgebras $\Lambda$ of index $|L:\Lambda|=m$ in $L$, where $|L:\Lambda|$ denotes the cardinality of  $L/\Lambda$. Define the \textit{subalgebra zeta function of $L$} as the Dirichlet generating series
\begin{align}\label{eq:uni.subalg}
		\zeta_{L}(s)=\sum_{m=1}^{\infty}a_{m}(L)m^{-s},
	\end{align}
where $s$ is a complex variable. 

Since the number of all $R$-submodules of $R^{d}$ of index $m$ is bounded polynomially as a function of $m$, the subalgebra zeta function of $L$ defines an analytic function in some complex right half-plane, providing asymptotics for the number of subalgebras of index less than $X$ as $X\rightarrow \infty$. For instance, consider $\Z^{d}$ as a $\Z$-algebra of rank $d$, equipped with null multiplication. It is known that for $d\geq1$, 
\begin{equation}\label{eq:Zd.uni}
    \zeta_{\Z^{d}}(s)=\zeta(s)\zeta(s-1)\cdots\zeta(s-(d-1)),
\end{equation}
where $\zeta(s)=\sum_{m=1}^{\infty}m^{-s}$ is the Riemann zeta function (\cite{LS} contains five different proofs for this). This tells us certain asymptotics for the number of finite-index sublattices of the integer lattice $\Z^d$. For instance, as $X\rightarrow \infty$ we have
\begin{align*}
    s_{X}(\Z^d):=\sum_{m<X}a_{m}(\Z^d)&=\#\{\textrm{sublattices of $\Z^d$ of index $m<X$}\}\\
    &=\frac{\zeta(d)\zeta(d-1)\cdots\zeta(2)}{d}X^{d}+O(X^{d-1}\log(X)).
\end{align*}
Assume now that 	$R=\mcO$ is the ring of integers of a number field $K$. For a (non-zero) prime ideal $\mfp\in\Spec(\mcO)$ we write $\mcO_{\mfp}$ for the completion of $\mcO$ at $\mfp$, a complete discrete valuation ring of characteristic zero and residue field
	$\mcO/\mfp$ of cardinality $q$ and characteristic $p$,
	say. Let $L(\Gri_{\mfp}):=L\tensor_\mcO\Gri_{\mfp}$. Primary decomposition yields the Euler product 
		\begin{equation*}
		\zeta_{L}(s)
		=\prod_{\mfp\in\Spec(\mcO)\setminus\{(0)\}}\zeta_{L(\Gri_{\mfp})}(s),\label{eq:euler}
	\end{equation*}
	expressing the ``global'' zeta function $\zeta_{L}(s)$ as an infinite
	product of ``local'' zeta function $\zeta_{L(\Gri_{\mfp})}(s)$. It is a deep result that each individual local subalgebra zeta function $\zeta_{L(\Gri_{\mfp})}(s)$ is a rational function in the parameter $q^{-s}$ (cf. \cite[Theorem 3.5]{GSS}). 

Zeta functions counting finite-index subalgebras or certain families of subobjects of an $R$-algebra have been extensively studied over the last few decades; see \cite{duS,duSW,KLL,KNV,Lee1,LV, Rossmann2} and references therein. On one direction, one studies the properties of local zeta function $\zeta_{L(\Gri_{\mfp})}(s)$. In \cite[Theorem 1.3]{duSG} and \cite[Theorem A]{Voll}, the authors proved that for a given $\mcO$-algebra $L$ of rank $d$, there are smooth quasi-projective varieties $V_i$ ($i\in I$, $I$ finite), defined over $K$, and rational functions $W_{i}(X,Y)\in\mathbb{Q}(X,Y)$ such that for almost all $\mfp\in\Spec(\mcO)$,
    \begin{equation}\label{eq:duSG}
	\zeta_{L(\Gri_{\mfp})}(s)=\sum_{i\in I}\left|\overline{V_i}(\mathbb{F}_q)\right|W_i(q,q^{-s}),
\end{equation}
where $\overline{V_i}$ denotes the reduction modulo $\mfp$ of a fixed $\Gri_{\mfp}$-model of $V_i$ and $\left|\overline{V_i}(\mathbb{F}_q)\right|$ denotes the number of $\Fq$-points of $\overline{V_i}$. 
\begin{rem}
Strictly speaking, \cite[Theorem 1.3]{duSG} and \cite[Theorem A]{Voll} only covers the case $K=\Q$, but the extension to an arbitrary number
field $K$ is straightforward. Please see \cite[Section 5]{Rossmann1} or \cite[Theorem~4.1]{Rossmann2} for details.
\end{rem}
It is of great interest to understand how $\zeta_{L(\Gri_{\mfp})}(s)$ varies with $\mfp$: called the \textit{uniformity} problem.
\begin{dfn}\label{def:uniform.zp}
		The (global) zeta function $\zeta_{L}(s)$ is \textit{finitely uniform} if there exist finitely many rational functions $W_{1}(X,Y),\ldots,W_{k}(X,Y)\in\Q(X,Y)$ for $k\in\mathbb{N}$ such that, for every (non-zero) prime ideal $\mfp\in\Spec(\mcO)$, 
		\[\zeta_{L(\Gri_{\mfp})}(s)=W_{i}(q,q^{-s})\]
		for some $i\in\{1,\ldots,k\}$. It is  \textit{uniform} if $k=1$ for all but finitely many $\mfp\in\Spec(\mcO)$, and \textit{non-uniform} if it is not finitely uniform.  
	\end{dfn}

Furthermore, in \cite[Theorem A]{Voll} Voll showed that for almost all $\mfp$, the local zeta function $\zeta_{L(\Gri_{\mfp})}(s)$ satisfies the following functional equation: 
\begin{equation}\label{eq:fun.eq.uni}  
    \left.\zeta_{L(\Gri_{\mfp})}(s)\right|_{q\rightarrow q^{-1}}=(-1)^{d}q^{\binom{d}{2}-ds}\zeta_{L(\Gri_{\mfp})}(s).
\end{equation}
\begin{rem}
  The operation $q \rightarrow q^{-1}$ in \eqref{eq:fun.eq.uni}
  requires some clarification. If a single rational
  function $W(X,Y)\in\Q(X,Y)$ exists such that the equality
  $\zeta_{L(\Gri_{\mfp})}(s) = W(q,q^{-s})$ holds for almost all $\mfp$, then the
  functional equation  \eqref{eq:fun.eq.uni} means that
  $$W(X^{-1},Y^{-1}) = (-1)^a X^b Y^{c} \, W(X,Y)$$ for
  suitable $a,b,c\in\N_0$.

  In general, the symmetry expressed in~\eqref{eq:fun.eq.uni} refers to Denef-type
    formulae \eqref{eq:duSG} for the zeta functions $\zeta_{L(\Gri_{\mfp})}(s)$. By the Weil conjectures, the numbers $\left|\overline{V_i}(\mathbb{F}_q)\right|$ may be written as alternating sums of Frobenius eigenvalues. The operation $q \mapsto q^{-1}$ inverts these eigenvalues and evaluates $W_{i}$ at $(q^{-1},q^{s})$. That this operation is
  well-defined (i.e.\ independent of the choice of Denef-type formula)
  follows from  \cite[Section~4]{RossmannMPCPS/18}. We refer to
  \cite[Remark~1.7]{VollIMRN/19} for further details.
\end{rem}

On another direction, similar to \eqref{eq:Zd.uni} a number of more refined questions about the distribution of  sublattices of $\Z^d$ and $R^{d}$ can be asked. In \cite{Petro}, Petrogradsky investigated the distribution of sublattices of $\Z^{d}$ whose \emph{cotypes} are given as a certain form. In \cite[Section 8]{Petro}, the author also extended this notion to $R^d$ in our setting, which we will record now.

For a given sublattice $\Lambda\subseteq R^{d}$ of finite index, one has a unique decomposition
\[R^{d}/\Lambda\cong\bigoplus_{\mfp_{i}}\left(\bigoplus_{j}\mcO/\mfp_{i}^{n_{i,j}}\right),\]
where each $\mfp_{i}$ is a non-zero prime ideal in $R$. As explained in \cite[Section 8]{Petro}, by collecting the largest powers of all primes into the first factor, the second largest powers of all primes into the second factor, etc., we obtain a unique decomposition
\[R^{d}/\Lambda\cong\bigoplus_{i=1}^{d}R/I_{i},\]
where $0\neq I_1\subseteq I_2\subseteq\cdots\subseteq I_d$. Denote $\alpha_{i}(\Lambda)=|R/I_{i}|$ for $i\in\{1,\ldots,d\}$, and let us write $\alpha(\Lambda):=(\alpha_{1}(\Lambda),\ldots,\alpha_{d}(\Lambda))$. We call $\alpha(\Lambda)$ the \emph{cotype} of $\Lambda$.
\begin{dfn}
  We define the  \emph{cotype zeta function of $R^{d}$} as the Dirichlet generating series
\begin{align*}
    \zeta_{R^{d}}^{\co}(\bss)=\zeta_{R^{d}}^{\co}(s_1,\ldots,s_{d})=\sum_{\Lambda\subseteq R^{d},|R^{d}/\Lambda|<\infty}\alpha_1(\Lambda)^{-s_1}\cdots\alpha_{d}(\Lambda)^{-s_{d}},
\end{align*}
 where each $s_{i}$ for $1\leq i \leq d$ is a complex variable and $\bss=(s_{1},\ldots,s_d)$. (In \cite{Petro} the author called $\zeta_{R^{d}}^{\co}(\bss)$ the \textit{multiple zeta function}.)
\end{dfn}

Let $R_{\mfp}$ be the completion of $R$ at $\mfp$. In \cite[Lemma 8.1]{Petro} the author also showed that 
\begin{align*}
    \zeta_{R^{d}}^{\co}(\bss)=\prod_{\mfp\in\Spec(R)\setminus\{(0)\}}\zeta_{R_{\mfp}^{d}}^{\co}(s).
\end{align*}

For a given $\Lambda$, the largest $i$ for which $\alpha_{i}(\Lambda)\neq1$ is called the \emph{corank} of $\Lambda$. By convention, $R^d$ has corank 0. A sublattice $\Lambda$ of corank 0 or 1 is called \emph{cocyclic}, i.e., $R^{d}/\Lambda$ is cyclic. Suppose we fix $R=\Z$ and let $N_{d}^{(m)}$ be the number of sublattices $\Lambda$ of $\Z^{d}$ of index less than $X$ such that $\Lambda$ has corank at most $m$. In \cite{Petro} and \cite{NS}, the authors showed that
\begin{equation*}
    N_{d}^{(1)}(X)\sim \frac{\theta_{d}}{d}X^{d},
\end{equation*}
as $X\rightarrow \infty$, where $\theta_{d}=\prod_{p,\,prime}\left(1+\frac{p^{d-1}-1}{p^{d+1}-p^{d}}\right)$. This allowed the authors to observe that the probability that a ``random'' sublattice of $\Z^{d}$ us cocyclic is about 85\% for large $d$. Chinta, Kaplan, and Koplewitz further extended this observation to $N_{d}^{(m)}$ for $1\leq m\leq d$ and observed interesting asymptotics on the distribution of lattices of corank $m$ (\cite[Theorem 1.1 and 1.2]{CKK}). 

\subsection{Main results and organization}
The aim of this paper is twofold. The first aim is to extend the notion of \textit{cotype} zeta functions from $R^{d}$ to an $R$-algebra $L$, and study their properties in a broader context. 

Let  $L$ be an $R$-algebra of rank $d$. Again, for a given sublattices $\Lambda\subseteq L$ there exists a unique decomposition
\[L/\Lambda\cong\bigoplus_{i=1}^{d}R/I_{i},\]
where $0\neq I_1\subseteq I_2\subseteq\cdots\subseteq I_d$. Denote $\alpha_{i}(\Lambda)=|R/I_{i}|$ for $i\in\{1,\ldots,d\}$ as before, and let us analogously define $\alpha(\Lambda):=(\alpha_{1}(\Lambda),\ldots,\alpha_{d}(\Lambda))$  

\begin{dfn}
  We define the  \emph{subalgebra cotype zeta function of $L$} as the Dirichlet generating series
\begin{align*}
    \zeta_{L}^{\co}(\bss)=\sum_{\Lambda\leq L,|L/\Lambda|<\infty}\alpha_1(\Lambda)^{-s_1}\cdots\alpha_{d}(\Lambda)^{-s_{d}},
\end{align*}
enumerating the finite-index subalgebras of $L$ of given cotypes.
\end{dfn}
Note that if $L$ is an $R$-algebra of rank $d$ with null multiplication, then counting subalgebras of $L$ with a given cotype is equivalent to counting sublattices of $R^{d}$ with a given cotype. Furthermore, for $\bss=(s,s,\ldots,s)$, by slight abuse of notation let us write the univariate subalgebra cotype zeta function as
\[\zeta_{L}^{\co}(s):=\zeta_{L}^{\co}(s,\ldots,s).\]
Note that this gives the original subalgebra zeta function \eqref{eq:uni.subalg}
\[\zeta_{L}^{\co}(s,\ldots,s)=\zeta_{L}(s).\]

Hence the zeta function $\zeta_{L}^{\co}(\bss)$ comprises both $\zeta_{L}(s)$ and $\zeta_{R^d}^{\co}(\bss)$. Developing a general theory for $\zeta_{L}^{\co}(\bss)$ therefore has clear implications for this area of study. In fact, throughout this article we will demonstrate how various theorems and explicit computations we record here generalize and unify previous results obtained for $\zeta_{L}(s)$ and $\zeta_{R^d}^{\co}(\bss)$. 

Assume now $R=\mcO$ is the ring of integers of a number field $K$, and write $\mcO_{\mfp}$ for the completion of $\mcO$ at $\mfp$ and $L(\Gri_{\mfp}):=L\tensor_\mcO\Gri_{\mfp}$ as above. Primary decomposition again yields the Euler product 
		\begin{equation*}
		\zeta_{L}^{\co}(s)
		=\prod_{\mfp\in\Spec(\mcO)\setminus\{(0)\}}\zeta_{L(\Gri_{\mfp})}^{\co}(s).
	\end{equation*}

In Section \ref{sec:fun.eq}, by appealing to the $\mfp$-adic machinery  developed in \cite{Voll}, we prove the following results. For $d\in\N$ write $\bfY=(Y_{i})_{1\leq i\leq d}$. 

\begin{thm}  \label{thm:fun.eq}
\begin{enumerate}
    \item There are smooth quasi-projective varieties $V_i$ ($i\in I$, $I$ finite), defined over $K$, and rational functions $W_{i}(X,\bfY)\in\mathbb{Q}(X,\bfY)$ such that for almost all $\mfp\in\Spec(\mcO)$,
    \begin{equation}\label{eq:duSG.co}
	\zeta_{L(\Gri_{\mfp})}^{\co}(\bss)=\sum_{i\in I}\left|\overline{V_i}(\mathbb{F}_q)\right|W_i(q,(q^{-s_{i}})_{1\leq i\leq d}),
\end{equation}
\item For all but finitely many prime ideals $\mfp$ of $\Gri$, we have
    \begin{equation}\label{eq:fun.eq.co}
    \left.\zeta_{L(\Gri_{\mfp})}^{\co}(\bss)\right|_{q\rightarrow q^{-1}}=(-1)^{d}q^{\binom{d}{2}-(s_1+\cdots+s_d)}\zeta_{L(\Gri_{\mfp})}^{\co}(\bss).
    \end{equation}
\end{enumerate}
\end{thm}

In particular, for $\bss=(s,\ldots,s)$ we obtain \cite[Theorem A]{Voll}.
\begin{rem}
  The finite collection of prime ideals $\mfp$ of $\Gri$ we are forced to
  exclude in \eqref{eq:fun.eq.uni} or \eqref{eq:fun.eq.co} are essentially those for
  which a chosen principalization of ideals of an algebraic variety exhibits bad reduction modulo~$\mfp$. While we know of no bounds on the size
  or shape of this finite set of ``bad'' prime ideals, it appears to be non-empty in general. 
\end{rem}

One of the prominent examples studied in the literature of subalgebra zeta functions of an $R$-algebra is that of $R$-Lie algebras, i.e., an $R$-algebra equipped with an antisymmetric bi-additive form (so-called ``Lie bracket'' $[-,-]:L\times L\rightarrow L$) satisfying the Jacobi identity. The second aim of this paper is to apply our method for $\Z$-Lie algebra $L$ of rank $d\leq 3$ and study certain asymptotics for the number of sublattices in $L$.  

In Section \ref{sec:dim3}, we obtain explicit formulas for subalgebra cotype zeta functions of $\Zp$-Lie algebras of rank 3, generalizing \cite[Theorem 1.1]{KV} in the context of cotype enumeration.
\begin{thm}\label{thm:cozeta.3dim}
Let $L(\Zp)$ be a 3-dimensional $\Zp$-Lie algebra. Then there is a ternary quadratic form $f(\bfx)\in\Zp[x_1,x_2,x_3]$, unique up to equivalence, such that, for $i\geq 0$
\begin{align*}
    \zeta_{p^{i}L}^{\co}(\boldsymbol{s})=\zeta_{\Z_{p}^3}^{\co}(\boldsymbol{s})-\frac{Z_{f}(s_{1}-2)(p^{2-s_1})^{i+1}(1+p^{1-s_1-s_2})(1-p^{1-2s_1})}{(1-p^{2-s_1})(1-p^{2-s_1-s_2})(1-p^{2-2s_1-s_2-s_3})(1-p^{-1})},
\end{align*}
where $Z_{f}(s)$ is Igusa's local zeta function associated to $f$.
\end{thm}
Note that for $\bss=(s,\ldots,s)$ we retain \cite[Theorem 1.1]{KV}. 
\begin{rem}
    As stated in \cite[p.4]{KV}, in general no simple identity is known which  relates $\zeta_{L}(s)$ (or $\zeta_{L}^{\co}(\boldsymbol{s})$) with that of $\zeta_{p^{i}L}(s)$ (or $\zeta_{p^{i}L}^{\co}(\boldsymbol{s})$) for $i\in\N_{0}$ yet.
\end{rem}
At the same time, Equation $\eqref{eq:duSG.co}$ in Theorem $\ref{thm:fun.eq}$ allows us to extend the definition of \emph{uniformity} to that of $\zeta_{L}^{\co}(\bss)$. 
\begin{dfn}\label{def:uniform.co}
		The  zeta function $\zeta_{L}^{\co}(\bss)$ is \textit{finitely uniform} if there exist finitely many rational functions \[W_{1}(X,(Y_{i})_{i\in[d]}),\ldots,W_{k}(X,(Y_{i})_{i\in[d]})\in\Q(X,(Y_{i})_{i\in[d]})\] for $k\in\mathbb{N}$ such that, for every (non-zero) prime ideal $\mfp\in\Spec(\mcO)$, 
		\[\zeta_{L(\Gri_{\mfp})}^{\co}(\bss)=W_{i}(q,(q^{-s_{i}})_{i\in[d]})\]
		for some $i\in\{1,\ldots,k\}$. It is  \textit{uniform} if $k=1$ for all but finitely many $\mfp\in\Spec(\mcO)$, and \textit{non-uniform} if it is not finitely uniform.  
	\end{dfn}
 Note that by substituting $s_1=s_2=\cdots=s_d=s$, this definition subsumes Definition \ref{def:uniform.zp}. In Section \ref{sec:dim3} we also prove the following result.
\begin{thm}\label{thm:dim3.uniformity}
    Let $L$ be a $\Z$-Lie algebra of rank 3. Then $\zeta_{L}(s)$ is always finitely uniform, and is uniform/finitely uniform if and only if $\zeta_{L}^{\co}(\bss)$ is uniform/finitely uniform.
\end{thm}


Furthermore, by computing $\zeta_{L(\Zp)}^{\co}(\bss)$ for \emph{all} primes including $p=2$, we compute the explicit global formula
\[\zeta_{L}^{\co}(\bss)=\prod_{p\,prime}\zeta_{L(\Zp)}^{\co}(\bss)\]
for certain $\Z$-Lie algebras $L$ of rank 3. This allows us to investigate the following asymptotic results related to the number of subalgebras of fixed corank. Let $N_{L}^{(m)}$ denote the number of subalgebras $\Lambda$ of $L$ of index less than $X$ such that $\Lambda$ has corank at most $m$, and let $P_{L}^{(m)}$ denotes the proportion of subalgebras of $L$ of corank at most $m$. In Section \ref{sec:explicit} we prove the following result:

\begin{thm}\label{thm:numeric}Let $L\in\{\Z^{3},H,\textrm{sl}_2(\Z),L_1,L_2\}$ be a $\Z$-Lie algebra of rank 3 as defined in Section \ref{sec:dim3}. As $X\rightarrow \infty$, we have
\begin{align*}
 P_{\Z^{3}}^{(1)}&\rightarrow 0.885,&P_{\Z^{3}}^{(2)}&\rightarrow 0.998,\\ 
    P_{H}^{(1)}&\rightarrow 0.492,&P_{H}^{(2)}&\rightarrow 0.975,\\   
    P_{\textrm{sl}_2(\Z)}^{(1)}&\rightarrow 0.488,&P_{\textrm{sl}_2(\Z)}^{(2)}&\rightarrow 0.974,\\   P_{L_1}^{(1)}&\rightarrow 0.492,&P_{L_1}^{(2)}&\rightarrow 0.975,\\ P_{L_2}^{(1)}&\rightarrow 0.482,&P_{L_2}^{(2)}&\rightarrow 0.970.
\end{align*}
\end{thm}
Here we view $\Z^{3}$ as an abelian $\Z$-Lie algebra, or equivalently, a $\Z$-algebra with null multiplication. 

Theorem \ref{thm:numeric} tells us many interesting observations. First, $P_{\Z^{3}}^{(1)}$ is compatible with the result in \cite{CKK} where the authors showed that most of the sublattices in $\Z^{d}$ are cocyclic, since any sublattices in $\Z^{3}$ are also subalgebras in $\Z^{3}$ with null multiplication. Second, if you compare $P_{\Z^{3}}^{(1)}$ with the rest of $P_{L}^{(1)}$, one sees that as soon as we introduce extra multiplication structure, the density of cocyclic subalgebras decreases so that only about half of the subalgebras are cocyclic. At the same time, it looks like being nilpotent, solvable, or simple does not really matter. 

Note that our results also answer the questions raised in \cite[Section 5.2]{CKK}, where the authors asked for a multivariate generalization of the subalgebra zeta function of the discrete Heisenberg group $H$ or $\textrm{sl}_2(\Z)$.

\subsection{Notation}
We write $\N$ for the natural numbers $\{1,2,\ldots,\}$. Given a subset $I\subseteq\N$, we write $I_{0}$ for $I\cup\{0\}$. For $m\in\N_{0}$, we set $[m]=\{1,2,\ldots,m\}$. 

Throughout, we write $\mcO$ for the ring of integers of a number field $K$, $\Gri_{\mfp}$ for the completion of  $\Gri$  at a nonzero prime ideal $\mfp$ of $\Gri$, $K_{\mfp}$ for the field of fractions of $\Gri_{\mfp}$, $q$ for the
	cardinality of the residue field of $\Gri_{\mfp}$,  $p$ for its residue
	characteristic, and $v_{\mfp}$ for the $\mfp$-adic valuation.	For a given $\mcO$-Lie algebra $L$, we write $L(\Gri_{\mfp}):= L\tensor_{\mcO}\Gri_{\mfp}$.  Whenever there is a presentation for an $R$-algebra or a group, we always assume that up to anti-symmetry, all other unlisted commutators are trivial. We write $t:=q^{-s}$, $t_{i}:=q^{-s_{i}}$, and $z_{i}:=\sum_{j\leq i}s_{j}$ for $i\in\N$, where $s$ and $s_{i}$ are complex variables.
\section{Preliminaries}\label{sec:prelim}
\subsection{Gaussian binomials, Igusa functions, and the local cotype zeta function of $\Z^{d}$}
For a variable $Y$ and integers $a,b\in\N_{0}$ with $a\geq b$, the associated \emph{Gaussian binomial} is
\[\binom{a}{b}_{Y}=\frac{\prod_{i=a-b+1}^{a}(1-Y^{i})}{\prod_{j=1}^{b}(1-Y^{j})}\in \Z[Y].\]
Given $d\in \N$ and a subset $I=\{i_{1},\ldots,i_{l}\}_{<}\subseteq[d-1]$, the associated \emph{Gaussian multinomial} is defined as
\begin{equation*}
    \binom{d}{I}_{Y}=\binom{d}{i_{l}}_{Y}\binom{i_{l}}{i_{l-1}}_{Y}\cdots\binom{i_{2}}{i_{1}}_{Y}\in\Z[Y].
\end{equation*}
\begin{dfn}[\cite{SV}, Definition 2.5]
Let $d\in\N$. Given variables $Y$ and $\boldsymbol{X}=(X_1,\ldots,X_{d})$, we define the \emph{Igusa function of degree $d$} as
\begin{equation*}
    I_{d}(Y;\boldsymbol{X}):=\frac{1}{1-X_{d}}\sum_{I\subseteq[d-1]}\binom{d}{I}_{Y}\prod_{i\in I}\frac{X_{i}}{1-X_{i}}\in\Q(Y,X_1,\ldots,X_{d}).
\end{equation*}
\end{dfn}
\begin{pro}[\cite{Petro}, Theorem 3.1. (5)] For $d\in\N$ and $i\in[d]$, let $z_{i}:=\sum_{j\leq i}s_{j}$. We have
\begin{equation*}
    \zeta_{\Zp^{d}}^{\co}(\bss)=I_{d}(p^{-1};(p^{i(d-i)-z_{i}})_{i\in[d]}).
\end{equation*}
\end{pro}
\begin{exm}\label{exm:cotype.Zd}
    We have
    \begin{align*}
        \zeta_{\Zp}^{\co}(\bss)&=\frac{1}{1-t_1},\\
        \zeta_{\Zp^{2}}^{\co}(\bss)&=\frac{1-t_1^2}{(1-t_1)(1-pt_1)(1-t_1t_2)},\\
          \zeta_{\Z_{p}^3}^{\co}(\bss)&=\frac{1+t_1+pt_1+t_1t_2+pt_1t_2+pt_1^2t_2}{(1-p^2t_1)(1-p^2t_1t_2)(1-t_1t_2t_3)}.
    \end{align*}
\end{exm}
An important feature of Igusa functions is that they satisfy a functional equation upon inversion of variables. It follows from \cite[Theorem 4]{SV} that, for all $d\in\N$, 
\begin{equation}\label{eq:fun.Igusa}
    I_{d}(Y^{-1};\bfX^{-1})=(-1)^{d}X_{d}Y^{\binom{d}{2}}I_{d}(Y;\bfX).
\end{equation}
Note that Equation \eqref{eq:fun.Igusa} is consistent with Theorem \ref{thm:fun.eq}. 
\subsection{Igusa's local zeta function}\label{subsec:Igusa}
Here we introduce the notion of Igusa's local zeta function that we use later in Section \ref{sec:dim3}. We refer the readers to \cite{Denef, Igusa} for general background on the theory. Let $f(\bfx)\in\Zp[x_1,\cdots,x_d]$. Igusa's local zeta function associated to $f$ is defined as the $p$-adic integral
\[Z_{f}(s):=\int_{\Zp^{d}}|f(\bfx)|_{p}^{s}\textup{d}\mu,\]
where $|\cdot|_{p}$ denotes the $p$-adic absolute value and $\textup{d}\mu$ stands for the additive Haar measure on $\Z_{p}^{d}$ normalized as $\mu(\Z_p^d ) = 1$.
For $m\in\N_{0}$, set
\[N_{m}:=\left|\{\bfx\in(\Z/p^{m}\Z)^{d}|f(\bfx)=0\}\right|.\]
A simple computation (cf. \cite[Section 2.1]{Denef}) shows that the Poincar\'{e} series
\[P_{f}(t):=\sum_{m=0}^{\infty}N_{m}(p^{-d}t)^{m}\]
is related to Igusa's local zeta function by the formula
\begin{equation*}\label{eq:igusa.poincare}
P_f(t)=\frac{1-tZ_{f}(s)}{1-t}.
\end{equation*}
As we see in Section \ref{sec:dim3}, counting subalgebras is related to counting solutions of polynomial equations in finite projective space. For $m\in\N_{0}$, let us define the affine cones
\begin{align*}
    W&:=\Z_{p}^{d}\,\setminus\, p\Z_{p}^{d},\\
    W(m)&:=(W+(p^{m}\Zp)^{d})/(p^{m}\Zp)^{d},
\end{align*}
and set 
\begin{equation*}
    N_{m}^{\star}:=\left|\left\{\bfx\in W(m)|f(\bfx)=0\right\}\right|.
\end{equation*}
Let
\begin{equation*}
    P_{f}^{\star}(t):=\sum_{m=0}^{\infty}N_{m}^{\star}(p^{-d}t)^m.
\end{equation*}
If $f$ is homogeneous of degree $n$, say, we have (cf. \cite[(1) on p. 1141]{DM})
\begin{equation*}
    Z_f(s)=\frac{1}{1-p^{-d-ns}}Z_{f}^{\star}(s),
\end{equation*}
where
\begin{equation*}
    Z_{f}^{\star}(s)=\int_{W}|f(\bfx)|_{p}^{s}\textup{d}\mu.
\end{equation*}
Let 
\[\mu_{m}^{\star}:=\mu(\{\bfx\in W:v_{p}(f(\bfx))=m\}).\]

In \cite{KV}, the authors showed that
\[\mu_{m}^{\star}=\frac{N_{m}^{\star}}{p^{nm}}-\frac{N_{m+1}^{\star}}{p^{n(m+1)}}-\delta_{m,0}p^{-d},\]
where 
\[\delta_{m,0}:=\begin{cases}
    1&m=0,\\
    0&m\neq0,
\end{cases}\]
denotes the Kronecker-delta. This allowed them to prove the following lemma:
\begin{lem}[Lemma 2.1, \cite{KV}]
    If $f$ is homogeneous, then
    \begin{equation*}
    P_{f}^{\star}(t)=\frac{1-p^{d}t-tZ_{f}^{\star}(s)}{1-t}.
    \end{equation*}
\end{lem}


\section{Local functional equations for cotype zeta functions of $\Op$-algebras}\label{sec:fun.eq}
\subsection{Lattices, matrices, and the subalgebra condition}
Our proof of Theorem \ref{thm:fun.eq} proceeds by adapting the $\mfp$-adic machinery developed in \cite{Voll}. 

Let $L$ be an $\mcO$-algebra of rank $d$ and write $L(\mcO_{\mfp}):=L\otimes_{\mcO}\mcO_{\mfp}$. We want to establish the functional equation for almost all zeta functions $\zeta_{L(\Gri_{\mfp})}^{\co}(\bss)$, enumerating $\Gri_{\mfp}$-sublattices $\Lambda$ of finite index in $L(\Gri_{\mfp})\cong\Gri_{\mfp}^{d}$ which are also subalgebras of $L(\Gri_{\mfp})$.

Let $\bse=(e_1,\ldots,e_d)$ be a basis of $L$ such that $L=\mcO e_1\oplus\cdots\oplus \mcO e_d$. Let $\Gamma=\GL_{d}(\Op)$. A full $\Op$-sublattice $\Lambda$ of $L(\Op)$ may be identified with a coset $\Gamma M$ for a matrix $M\in\GL_{d}(K_{\mfp})\cap \Mat_{d}(\Op)$, whose rows encode the coordinates with respect to $\bse$. Let $\pi\in\mfp$ be a uniformizer of $\mfp$ in $\Op$. By the elementary divisor theorem, for a given $\Lambda$ there exists $I=\{i_{1},\ldots,i_{l}\}_{<}\subseteq[d-1]$, $r_0\in\N_{0}$, and $\boldsymbol{r}=(r_{i_{1}},\ldots,r_{i_{l}})\in\N^{l}$, all uniquely determined by $\Lambda$, and $\alpha\in \Gamma$ such that $M=D\alpha^{-1}$, where
\begin{align*}
    D&=\pi^{r_0}\diag(\pi^{r_1+\cdots+r_{d-1}},\pi^{r_2+\cdots+r_{d-1}},\ldots,1)\\
    &=\pi^{r_{0}}\diag((\pi^{\sum_{\iota\in I}r_{\iota}})^{(i_1)},(\pi^{\sum_{\iota\in I\setminus\{i_1\}}r_{\iota}})^{(i_2-i_1)},\ldots,(\pi^{r_{i_{l}}})^{(i_{l}-i_{l-1})},1^{(d-i_{l})}).
\end{align*}
For $i\in[d]$, let $z_{i}:=\sum_{j\leq i}s_{j}$. One can check that 
\begin{align*}
\alpha_1(\Lambda)^{-s_1}\cdots\alpha_{d}(\Lambda)^{-s_d}
&=q^{-\sum_{\iota\in I}r_{\iota}\sum_{i\leq \iota}s_{i}}\\
&=q^{-\sum_{\iota\in I}r_{\iota}z_{\iota}}
\end{align*}

We say the homothety class $[\Lambda]=\{x\Lambda\mid x\in K_{\mfp}^{*}\}$ has type $(I,\boldsymbol{r})$ and write $\nu([\Lambda])=(I,\bsr)$.  Note that every homothety class $[\Lambda]$ of $\Lambda$ in $L(\Gri_{\mfp})$ contains a unique ($\subseteq$-) maximal subalgebra $\Lambda_{0}\leq L(\Gri_{\mfp})$, and the subalgebras in this class are exactly the multiples $q^{m}\Lambda_{0}$ for $m\in\N_{0}$. Hence we have
\begin{equation}\label{eq:homothety}
    \zeta_{L(\Gri_{\mfp})}^{\co}(\bss)=\frac{1}{1-q^{-z_{d}}}A(\bss),
\end{equation}
where
\begin{align*}
A(\bss)=\sum_{[\Lambda]} \alpha_1(\Lambda_0)^{-s_1} \alpha_2(\Lambda_0)^{-s_2}\cdots\alpha_d(\Lambda_0)^{-s_d}.
\end{align*}
To prove Theorem \ref{thm:fun.eq} it suffices to show that
\begin{equation}\label{eq:A.fun.eq}
     \left.A(\bss)\right|_{q\rightarrow q^{-1}}=(-1)^{d-1}q^{\binom{d}{2}}A(\bss).
\end{equation}
We achieve this by showing that $A(\bss)$ is expressible in terms of certain $\mfp$-adic integrals. Consider the $d\times d$ matrix of $\mcO$-linear forms 
\[\mcR(\bfy)=\left(\sum_{k\in[d]}\lambda_{ij}^{k}y_{k}\right)_{ij}\in\Mat_{d}(\mcO[\bfy]),\]
where $\lambda_{ij}^{k}$ are the structure constants of $L$ with respect to the chosen basis $\bse$ such that $e_ie_j=\sum_{k\in[d]}\lambda_{ij}^{k}y_k$. For $i\in[d]$ let $C_i$ denote the matrix of the linear map given by right-multiplication with the generator $e_{i}$. Denote by $\bsm_{i}$ the $i$-th row of $M$. We have $\Lambda\leq L(\Op)$ if and only if
\begin{equation}\label{eq:sub.cond1}
    \forall i,\,j\in[d]:\,\bsm_{i}\sum_{r\in[d]}C_{r}m_{jr}\in\langle \bsm_{k}|k\in[d]\rangle_{\Op}.
\end{equation}
Write $M=D\alpha^{-1}$ as above. Let $\alpha([i])$ denote the $i$-th column of the matrix $\alpha$ and $D_{ii}$ denote the $i$-th diagonal entry of $D$. Equation \eqref{eq:sub.cond1} is equivalent to 
\begin{equation}\label{eq:sub.cond.2}
\forall i\in[d]:\,D\mcR_{(i)}(\alpha)D\equiv0\bmod D_{ii},
\end{equation}
where $\mcR_{(i)}(\alpha):=\alpha^{-1}\mcR(\alpha[i])(\alpha^{-1})^{t}$. 
For $i,r,s\in[d]$, let \[v_{irs}(\alpha):=\min\left\{v_{\mfp}\left(\mcR_{(\iota)}(\alpha))_{\rho\sigma}\right)|\iota\leq i, \rho\geq r,\sigma\geq s\right\}\]
and 
\[m([\Lambda]):=\min\left\{\sum_{\iota\in I}r_{\iota},\sum_{s\leq \iota\in I}r_{\iota}+\sum_{r\leq\iota\in I}r_{\iota}+\sum_{i>\iota\in I}r_{\iota}+v_{irs}(\alpha)|(i,r,s)\in[d]^{3}\right\}.\]

Equation \eqref{eq:sub.cond.2} can be reformulated as (cf. \cite[Equation 28]{Voll})
\begin{equation*}\label{eq:m.widetilde}
    r_{0}\geq\sum_{\iota\in I}r_{\iota}-m([\Lambda])=:\widetilde{m}([\Lambda]).
\end{equation*}
By \eqref{eq:homothety} it  suffices to compute
\begin{align*}
    A(\bss)&=\sum_{[\Lambda]} \alpha_1(\Lambda_0)^{-s_1} \alpha_2(\Lambda_0)^{-s_2}\cdots\alpha_d(\Lambda_0)^{-s_d}\\
    &=\sum_{I\subseteq[d-1]}\sum_{\nu([\Lambda])=I}\alpha_1(\Lambda_0)^{-s_1} \alpha_2(\Lambda_0)^{-s_2}\cdots\alpha_d(\Lambda_0)^{-s_d}.
\end{align*}
Let $A_{I}(\bss):=\sum_{\nu([\Lambda])=I}\alpha_1(\Lambda_0)^{-s_1} \alpha_2(\Lambda_0)^{-s_2}\cdots\alpha_d(\Lambda_0)^{-s_d}$.  For a fixed $I=\{i_1,\ldots,i_{l}\}_{<}\subseteq[d-1]$, we set
\[\mcN_{I,\bsr,m}:=|\{[\Lambda]\mid\nu([\Lambda])=(I,\bsr),\,m([\Lambda])=m\}|.\] We have
\begin{align*}
    A_{I}(\bss)&=\sum_{\bsr=(r_{i_{1}},\ldots,r_{i_{l}})\in\N^{l}}q^{-\sum_{\iota\in I}r_{\iota}z_{\iota}}\sum_{m\in\N_{0}}\mcN_{I,\bsr,m}q^{-\widetilde{m}([\Lambda])z_d}\\
    &=\sum_{\bsr=(r_{i_{1}},\ldots,r_{i_{l}})\in\N^{l}}q^{-\sum_{\iota\in I}r_{\iota}(z_{\iota}+z_{d})}\sum_{m\in\N_{0}}\mcN_{I,\bsr,m}q^{mz_{d}}\notag.
\end{align*}
\subsection{$\mfp$-adic integration}
Analogous to \cite[Section 3.1]{Voll}, to establish the functional equation \eqref{eq:A.fun.eq} we express  $A_{I}(\bss)$  in terms of the $\mfp$-adic integrals of the form \cite[Equation (6)]{Voll} that satisfy the hypotheses of \cite[Theorem 2.3]{Voll}.

To this end we define, for $i,r\in[d]$, sets of polynomials (cf. \cite[Equation (31)]{Voll})
\begin{align*}
\bff_{irs}(\bfy)&=\{(\mcR_{\iota}(\bfy))_{\rho\sigma}|\iota\leq i,\rho\geq r, \sigma\geq s\},\,(i,r,s)\in[d]^{3}\\
\bfg_{d,I}(\bfx,\bfy)&=\left\{\prod_{\iota\in I}x_{\iota}\right\}\cup\bigcup_{(i,r,s)\in[d]^{3}}\left(\prod_{\iota\in I}x_{\iota}^{\delta_{\iota\geq r}+\delta_{\iota\geq s}+\delta_{\iota<i}}\right)\bff_{irs}(\bfy),
\end{align*}
and, for $\kappa\in[d-1]$,
\[\bfg_{\kappa,I}(\bfx,\bfy)=\left\{\prod_{\iota\in I}x_{\iota}^{\delta_{\iota\kappa}}\right\}.\]

With this data define the $\mfp$-adic integral

\begin{align*}
Z_{I}(\bss)&=Z_{I}((s_{\iota})_{\iota\in I},s_{d}):=\int_{\mfp^{|I|}\times\Gamma}\prod_{\kappa\in[d]}||\bfg_{\kappa,I}||^{s_{\kappa}}|\tud\bfx_{I}||\tud\bfy|,
\end{align*}

where $||\cdot||$ denotes the $\mfp$-adic (maximum) norm and $|\textup{d}\bfx_{I}||\tud\bfy|$ denotes the additive Haar measure, normalized such that the domain of integration has measure $q^{|I|}\mu(\Gamma)$, where $\mu(\Gamma)=\prod_{i=1}^{d}(1-q^{-i})$. By design the integral $Z_{I}(\bss)$ is precisely of the form \cite[Equation (6)]{Voll}. By omitting at most finitely many primes, we may assume that the assumptions of \cite[Theorem 2.2]{Voll} are satisfied. This implies that the normalized integrals
\begin{equation*}
    \widetilde{Z_{I}}(\bss):=\frac{Z_{I}(\bss)}{(1-q^{-1})^{|I|}\mu(\Gamma)}
\end{equation*}
(cf. \cite[Equation 10]{Voll}) satisfy the ``inversion properties'' in \cite[Theorem 2.3]{Voll}. Thus the sum
\begin{equation}\label{eq:Z.tilde}
    \widetilde{Z}(\bss):=\sum_{I\subseteq[d-1]}\binom{d}{I}_{q^{-1}}\widetilde{Z_{I}}(\bss)
\end{equation}
(cf. \cite[Equation (16)]{Voll}) satisfies the functional equation
\begin{equation}\label{eq:Z.fun.eq}
       \left.\widetilde{Z}(\bss)\right|_{q\rightarrow q^{-1}}=(-1)^{d-1}q^{\binom{d}{2}}\widetilde{Z}(\bss).
\end{equation}

Finally, one needs to show that for each $I\subseteq[d-1]$ the generating function $A_{I}(\bss)$ is obtainable from the $\mfp$-adic integral $Z_{I}(\bss)$ by a suitable specialization of the variables $\bss$. Note that (cf. \cite[Equation (32)]{Voll}

\begin{align*}
Z_{I}(\bss):=Z_{I}((s_{\iota})_{\iota\in I},s_{d})=\sum_{\bsr\in\N^{l}}q^{-\sum_{\iota\in I}s_{\iota}r_{\iota}}\sum_{m\in\N_{0}}\mu_{I,r,m}q^{-s_{d}m},
\end{align*}
where
\begin{equation*}
\mu_{I,r,m}:=\mu\left\{(\bfx,\bfy)\in\mfp^{|I|}\times\Gamma\mid\forall \iota\in I:v_{\mfp}(x_{\iota})=r_{\iota},m(\bfx,\bfy)=m\right\},
\end{equation*}
with
\begin{equation*}
    m(\bfx,\bfy)   :=\min\left\{\sum_{\iota\in I}r_{\iota},\sum_{s\leq \iota\in I}v_{\mfp}(x_{\iota})+\sum_{r\leq\iota\in I}v_{\mfp}(x_{\iota})+\sum_{i>\iota\in I}v_{\mfp}(x_{\iota})+v_{irs}(\bfy)|(i,r,s)\in[d]^{3}\right\}
\end{equation*}
The numbers $\mu_{I,r,m}$ are closely related to the natural numbers $\mcN_{I,r,m}$ we are looking to control.
\begin{lem}
    \begin{equation*}
        \mcN_{I,r,m}=\frac{\binom{d}{I}_{q^{-1}}}{(1-q^{-1})^{|I|}\mu(\Gamma)}\mu_{I,r,m}q^{\sum_{\iota\in I}r_{\iota}(\iota(d-\iota)+1)}.
    \end{equation*}
\end{lem}
\begin{proof}
    Analogous to \cite[Lemma 3.1]{Voll}.
\end{proof}
Hence we get
\begin{align*}
    A_{I}(\bss)
    &=\sum_{\bsr=(r_{i_{1}},\ldots,r_{i_{l}})\in\N^{l}}q^{-\sum_{\iota\in I}r_{\iota}(z_{\iota}+z_{d})}\sum_{m\in\N_{0}}\mcN_{I,\bsr,m}q^{mz_{d}}\\
    &=\frac{\binom{d}{I}_{q^{-1}}}{(1-q^{-1})^{|I|}\mu(\Gamma)}\sum_{\bsr=(r_{i_{1}},\ldots,r_{i_{l}})\in\N^{l}}q^{-\sum_{\iota\in I}r_{\iota}(z_{\iota}+z_{d}-\iota(d-\iota)-1)}\sum_{m\in\N_{0}}\mu_{I,\bsr,m}q^{mz_{d}}\\
    &=\frac{\binom{d}{I}_{q^{-1}}}{(1-q^{-1})^{|I|}\mu(\Gamma)}Z_{I}((z_{\iota}+z_{d}-\iota(d-\iota)-1)_{\iota\in I},-z_{d})\\
    &=\binom{d}{I}_{q^{-1}}\widetilde{Z_{I}}((z_{\iota}+z_{d}-\iota(d-\iota)-1)_{\iota\in I},-z_{d}).
\end{align*}
Since $A(\bss)=\sum_{I\subseteq[d-1]}A_{I}(\bss)$,  Theorem \ref{thm:fun.eq} now follows from Equation \eqref{eq:Z.tilde},  \eqref{eq:Z.fun.eq}, and \cite[Theorem 2.2]{Voll}.

\section{Cotype zeta functions of $\Z$-Lie algebras of dimension $d\leq3$}\label{sec:dim3}
 Let $L$ be a $\Z$-Lie algebra of rank $d$ and write $L_{p}:=L(\Zp)=L\otimes_{\Z}\Zp$. Recall that we have
\begin{align*}
    \zeta_{L}^{\co}(\bss)=\prod_{p,\,prime}\zeta_{L_p}^{\co}(\bss).
\end{align*}
In this section, we explicitly compute $\zeta_{L}^{\co}(\bss)$ and $\zeta_{L_p}^{\co}(\bss)$ for $d\leq3$ using the framework explained in Section \ref{sec:fun.eq}, which is also compatible with \cite{KV}.  

For $d=1$, there is only one isomorphism class of $\Z$-Lie algebra of rank $1$, namely $\Z$ with null multiplication, and we get
\begin{align*}
    \zeta_{L}^{\co}(\bss)=\zeta(s_1).
\end{align*}

For $d=2$, it is known that there is one isomorphism class of abelian $\Z$-Lie algebra
\[L_{0}=\langle e_1,e_2\rangle\cong\Z^2,\]
and infinite isomorphism classes of non-abelian $\Z$-Lie algebra
\[L_{1,k}=\langle e_1,e_2|[e_1,e_2]=ke_2\rangle\]
for $k\in\Z_{>0}$, where $L_{1,k}\cong L_{1,k'}$ if and only if $k=k'$.

\begin{pro}
    Let $L$ be a $\Z$-Lie algebra of rank 2. Then we have
    \begin{align*}
    \zeta_{L}^{\co}(\bss)=\zeta_{\Z^{2}}^{\co}(\bss)=\zeta(s_1)\zeta(s_1-1)\zeta(s_1+s_2)\zeta^{-1}(2s_1).
\end{align*}
\end{pro}
\begin{proof}
    If $L\cong L_{0}$ then it follows trivially. Suppose $L\cong L_{1,k}$ for any $k\in\Z_{>0}$.  We show that any finite index $\Z$ submodule of $L$ becomes a subalgebra. For any elements $v,w \in L$, we have
    \begin{align*}
        [v,w] &= [v_1 e_1 + v_2 e_2 , w_1 e_1 + w_2 e_2 ]\\
        &= (v_1 w_2-w_2 v_1)ke_2  = -kw_1(v_1 e_1+v_2 e_2)+kv_1(w_1 e_1+ w_2e_2) \in \Z v + \Z w,
    \end{align*}
so it follows that $M \le L$ implies $[M,M] \subseteq M$. Thus $\zeta_{L}^{\co}(\bss) = \zeta_{\Z^2}^{\co}(\bss)$.
\end{proof}

For $d=3$, instead of computing $ \zeta_{L}^{\co}(\bss)$ directly, we concentrate on $\zeta_{L_p}^{\co}(\bss)$.  Recall (Example \ref{exm:cotype.Zd}) the formula
\begin{align*}
    \zeta_{\Z_{p}^3}^{\co}(\boldsymbol{s})=\frac{1+t_1+pt_1+t_1t_2+pt_1t_2+pt_1^2t_2}{(1-p^2t_1)(1-p^2t_1t_2)(1-t_1t_2t_3)}.
\end{align*}

For $\Lambda\leq L_p$, let us denote by
\[
(\alpha_1(\Lambda), \alpha_2(\Lambda), \alpha_3(\Lambda)) = (p^{r_0+r_1+r_2}, p^{r_0 + r_2}, p^{r_0})
\]
the elementary divisors of $\Lambda$. As explained in Section \ref{sec:fun.eq}, a coset $\Gamma M$ corresponding to $\Lambda$ has a representative of the form
\[
\Gamma M =\Gamma D \alpha^{-1}, \quad D = p^{r_0} \mathrm{diag}(p^{r_1+r_2}, p^{r_2}, 1), \quad \alpha \in \Gamma.
\]

We get
\begin{align*}
    \zeta_{L_p}^{\co}(\bss)&=\sum_{\Lambda \leq L_p} \alpha_1(\Lambda)^{-s_1} \alpha_2(\Lambda)^{-s_2}\alpha_3(\Lambda)^{-s_3}\\
    &=\frac{1}{1-p^{-(s_1+s_2+s_3)} }A(\bss),
\end{align*}
where $A(\bss)=\sum_{[\Lambda]} \alpha_1(\Lambda_0)^{-s_1} \alpha_2(\Lambda_0)^{-s_2}\alpha_3(\Lambda_0)^{-s_3}$. Denote $\nu([\Lambda]) = \{ i \in \{1,2 \} | r_i \neq 0 \}$, and for $I \subseteq \{1,2\}$, let
\[A_{I}(\bss):=\sum_{\nu([\Lambda])=I}\alpha_1(\Lambda_0)^{-s_1} \alpha_2(\Lambda_0)^{-s_2}\alpha_3(\Lambda_0)^{-s_3}\]
as before.

Let $\lambda_{ij}^{k}$ be the structure constants of the $\Zp$-Lie algebra $L_p$ with respect to the given $\Z_p$-basis $(e_1, e_2, e_3)$, where
\[[e_i,e_j]=\sum_{k=1}^{3}\lambda_{ij}^{k}e_k\]
for $i,j,k\in[3]$, and consider
\[
\mathcal{A}_{L_p} = \begin{pmatrix} \lambda_{23}^1 & \lambda_{23}^2 & \lambda_{23}^3 \\ \lambda_{31}^1 & \lambda_{31}^2 & \lambda_{31}^3 \\ \lambda_{12}^1 & \lambda_{12}^2 & \lambda_{12}^3 \end{pmatrix}.
\]

Entries of $\mathcal{A}_{L_p}$ are exactly the nontrivial entries of $\lambda_{ij}^{k}$
such that the Lie bracket $[-,-] : L_p \times L_p \rightarrow L_p$ becomes antisymmetric, so this information determines $L_p$. We also note that $\mathcal{A}_{L_p}$ is a `natural object' -- $\mathcal{A}_{L_p}$ is the coordinate matrix of the bilinear form
\[
\alpha_{L_p} : \wedge^2 L_p \times \wedge^2L_p \rightarrow \wedge^3 L_p, \quad \alpha_{L_p}(x\wedge y, z\wedge w) = [x,y]\wedge z \wedge w
\]
with respect to basis $(e_2 \wedge e_3, e_3 \wedge e_1, e_1 \wedge e_2)$ of $\wedge^2 L_p$. 
This interpretation of $\mathcal{A}_{L_p}$ allows us to describe the transformation law for $\mathcal{A}_{L_p}$ with respect to a basis change as follows. For a matrix $P = \begin{pmatrix} p_{11} & p_{12} & p_{13} \\ p_{21} & p_{22} & p_{23} \\ p_{31}& p_{32} & p_{33} \end{pmatrix} \in \GL_3(\Qp)$, denote $P L_p$ as a same $\Qp$-Lie algebra $L_p$ with basis $(p_{11} e_1+p_{12}e_2+p_{13}e_3,\, p_{21} e_1+p_{22}e_2+p_{23}e_3, \,p_{31}e_1+p_{32}e_2+p_{33}e_3)$. Then the coefficient matrix $\mathcal{A}_{PL_p}$ for $PL_p$ is given as
\[
\mathcal{A}_{PL_p} = \det(P) (P^t)^{-1} \mathcal{A}_{L_p} P^{-1}.
\]

\begin{pro}
    A lattice $\Lambda$ corresponding to $\Gamma M = \Gamma D \alpha^{-1}$ is a subalgebra of $L_p$ if and only if
\[
v_p \left(( \alpha e_1)^t \mathcal{A}_{L_p} (\alpha e_1) \right) \ge -r_0 + r_1.
\]
\end{pro}

\begin{proof}
The lattice $\Gamma M$ becomes a $\Zp$-subalgebra of $L_p$ if and only if all entries of $\mathcal{A}_{ML_p}$ are in $\Z_p$. For $M = \gamma D \alpha^{-1}$ for $\gamma , \alpha^{-1} \in \Gamma$, this condition is equivalent to
\[
\det(M) (M^t)^{-1} \mathcal{A}_{L_p} M^{-1} = 
\det(D) (\gamma^{t})^{-1} D^{-1} \alpha^t \mathcal{A}_{L_p} \alpha D \gamma^{-1} \in \mathrm{Mat}_{3 \times 3} (\Z_p),
\]
which is then equivalent to
\[
\det(D) D^{-1} \alpha^t \mathcal{A}_{L_p} \alpha D^{-1} \in \mathrm{Mat}_{3 \times 3}(\Z_p).
\]
For $D = \mathrm{diag}(p^{r_0+r_1+r_2}, p^{r_0+r_2}, p^{r_0})$, this condition is equivalent to
\[
\alpha^t \mathcal{A}_{L_p} \alpha \in \begin{pmatrix} p^{-r_0 + r_1} \Z_p & p^{-r_0} \Z_p & p^{-r_0 - r_2} \Z_p \\
p^{-r_0} \Z_p & p^{-r_0 - r_1} \Z_p & p^{-r_0-r_1 -r_2} \Z_p \\ p^{-r_0 - r_2} \Z_p & p^{-r_0 - r_1 - r_2}\Z_p & p^{-r_0 - r_1 - 2r_2}\Z_p
\end{pmatrix}.
\]
As all entries of $\alpha^t \mathcal{A}_{L_p} \alpha$ are in $\Z_p$, the only relevant condition is
\begin{equation}\label{eq:sub.cond}
(\alpha e_1)^t \mathcal{A}_{L_p} (\alpha e_1) \in p^{-r_0 + r_1} \Z_p.
\end{equation}
\end{proof}
\begin{rem}
Note that condition \eqref{eq:sub.cond}  is essentially  equivalent to the (SUB) condition in \cite{KV}, where being multivariate or univariate does not affect the condition $\Lambda$ being a subalgebra of $L_p$. Like \cite{KV}, we need to find the smallest $r_{0}$ satisfying \eqref{eq:sub.cond}.
\end{rem}
Define
\[
f(\bfx) = \bfx^t \mathcal{A}_{L_p} \bfx.
\]

We obtain the following results, a  multivariate generalization of \cite[Section 3]{KV}.

\begin{pro}\label{pro:A_I}
We have
\begin{align*}
A_{\emptyset}(\bss) &= 1, \\
A_{\{2\}}(\bss)  &= \sum_{r_2=1}^{\infty} (p^{-2}+p^{-1}+1) p^{2r_2} p^{-s_1r_2-s_2r_2},\\
A_{\{1\}}(\bss) &= \sum_{r_1=1}^{\infty} \sum_{m=0}^{\infty} N_{r_1,m} p^{-s_1(2r_1-m)-s_2(r_1-m)-s_3(r_1-m)},\\
A_{\{1,2\}}(\bss) &= \sum_{r_2 = 1}^{\infty} (p^{-1}+1)p^{2r_2} \sum_{r_1=1}^{\infty} \sum_{m=0}^{\infty} N_{r_1,m} p^{-s_1(2r_1+r_2-m)-s_2(r_1+r_2-m)-s_3(r_1-m))} \\
&= (p^{-1} + 1) \frac{ p^{2 - s_1 - s_2}}{1 - p^{2-s_1 - s_2}} A_{\{1\}} (\bss),
\end{align*}
where
\[
N_{r,m} = | \{ \bfx = (x_1 : x_2 : x_3) \in \mathbb{P}^2 (\Z / p^r \Z) \mid \min(r, v_p(f(\bfx))) = m \} |
.\]
\end{pro}

\begin{proof}
For $\Lambda = \Gamma M = \Gamma D \alpha^{-1}$ where $D = \mathrm{diag}(p^{r_0+r_1+r_2},p^{r_0+r_2},p^{r_0})$ and $\alpha \in \Gamma$, we have
\[
v_p(f(\alpha e_1)) \ge -r_0 + r_1 \quad \Leftrightarrow \quad  r_0 \ge \max( r_1 - v_p(f(\alpha e_1)),0)
\]
so $\Lambda_0$ is given as $\Gamma D_0 \alpha^{-1}$ for
\[
D_0 = p^{r_0'}\mathrm{diag}(p^{r_1+r_2},p^{r_2},1), \quad r_0'=\max( r_1 - v_p(f(\alpha e_1)),0).
\]
Denote
\[
\max( r_1 - v_p(f(\alpha e_1)),0)= r_1 - m([\Lambda])
\]
which is equivalent to that
\[
m([\Lambda]) = r_1 - \max(r_1 - v_p(f(\alpha e_1)),0) = \min(v_p(f(\alpha e_1), r_1).
\]
Let $R_I$ to be the set of $(r_1, r_2)$ such that $r_i >0$ if and only if $i \in I$. Then
\begin{align*}
A_I(\bss) &= \sum_{r_1, r_2 \in R_I} \sum_{\nu([\Lambda])=(I,{\bf r})} \alpha_1(\Lambda_0)^{-s_1} \alpha_2(\Lambda_0)^{-s_2} \alpha_3(\Lambda_0)^{-s_3} \\
&= \sum_{r_1, r_2 \in R_I} \sum_{\nu([\Lambda])=(I,{\bf r})} p^{-(s_1+s_2+s_3)(r_1-m([\Lambda]))}p^{-s_1 (r_1+r_2)-s_2 r_2} \\
&= \sum_{r_1,r_2 \in R_I} \sum_{m=0}^{\infty} p^{-s_1(2r_1+r_2-m)-s_2(r_1+r_2-m)-s_3(r_1-m)}
\left|\{[\Lambda]: \nu([\Lambda])=(I,{\bf r}) , \min(r_1,v_p(f(\alpha e_1)))=m \}\right|.
\end{align*}

\noindent \textbf{Case $I=\emptyset$}: If $I =\emptyset$, meaning $r_1=0$, then we always have $\min(r_1,v_p(f(\alpha e_1)))=0$ so we add over all homothety classes with $\nu([\Lambda])=(I,{\bf r})$. Since there is only one unique homothety class os type ${\bf r}=\{0,0\}$, we obtain $A_{\emptyset}(\bss)= 1$.\\

\noindent \textbf{Case $I=\{2\}$}: If $I =\{2\}$, then again $r_1=0$ and we add over all homothety classes with $\nu([\Lambda])=(I,{\bf r})$. One can easily check that 

\[
A_{\{2\}}(\bss) = (p^{-2}+p^{-1}+1) \frac{p^{2 - s_1 - s_2}}{1 - p^{2 - s_1 - s_2}}.
\]

\noindent \textbf{Case $I=\{1\}$}:
It suffices to show that
\[
|\{[\Lambda]: \nu([\Lambda])=(I,(r_1, 0)) , \min(r_1,v_p(f(\alpha e_1)))=m \} | = N_{r_1, m}.
\]
A homothety class $[\Lambda] = [\Gamma D \alpha^{-1}]$ corresponds to elements $\alpha \in \Gamma$ up to right multiplication by $\Gamma_{\bf r} = \mathrm {Stab}_{\Gamma} (\Gamma D)$, or equivalently left coset $\alpha \Gamma_{\bf r} \in \Gamma / \Gamma_{\bf r}$.

We can check (cf. \cite[Section 3]{KV}) that
\begin{equation*}
    \Gamma_{(r_1,0)}=\begin{pmatrix}
        \Zp^{\times}&\begin{matrix}
        \Zp&\Zp\end{matrix}\\
        \begin{matrix}p^{r_1}\Zp\\p^{r_1}\Zp\end{matrix}&\GL_{2}(\Zp)
    \end{pmatrix}.
\end{equation*}
Consider the map $\pi_{r_1} : \Gamma \rightarrow \mathbb{P}^2(\Z/p^{r_1} \Z)$ where $\pi(\alpha)$ is defined as image of the first column $\alpha e_1$ in $\mathbb{P}^2(\Z/p^{r_1} \Z)$. Its kernel is exactly given as $\Gamma_{(r_1,0)}$, so it induces an one-to-one correspondence between $\Gamma / \Gamma_{(r_1, 0)}$ and $\mathbb{P}^2(\Z/p^r \Z)$. As we are counting representatives $\alpha^{-1}$ such that $\min\left( r, v_p(f(\alpha e_1)) \right)=m$, our statement follows.\\

\noindent \textbf{Case $I=\{1,2\}$}: It suffices to show that
\[
|\{[\Lambda]: \nu([\Lambda])=(I,(r_1, r_2)) , \min(r_1,v_p(f(\alpha e_1)))=m \} | = (p^{-1}+1)p^{2r_2} N_{r_1, m}.
\]
We can check that
\begin{equation*}
    \Gamma_{(r_1,r_2)}=\begin{pmatrix}
        \Zp^{\times}&\Zp&\Zp\\
        p^{r_1}\Zp&\Zp^{\times}&\Zp\\
        p^{r_1+r_2}\Zp&p^{r_2}\Zp&\Zp^{\times}
    \end{pmatrix}
\end{equation*}
and
\[
|\Gamma_{(r_1,0)} : \Gamma_{(r_1,r_2)}| = (p^{-1}+1)p^{2r_2}.
\]
We are counting representatives $\alpha^{-1}$ such that $\min\left( r, v_p(f(\alpha e_1)) \right)=m$, but now those are representatives of left $\Gamma_{(r_1,r_2)}$-cosets. Thus the number is $|\Gamma_{(r_1,0)} : \Gamma_{(r_1,r_2)}|$ times $N_{r_1, m}$, giving
\[A_{\{1,2\}}(\bss)=(p^{-1}+1)\frac{p^{2-s_1-s_2}}{1-p^{2-s_1-s_2}}A_{\{1\}}(\bss).\]

\end{proof}

We now use the machinery introduced in Section \ref{subsec:Igusa} to develop a formula for the Dirichlet series  $A_{\{1\}}(\bss)$. Note that this is in a very similar vein to \cite{KV}, and we keep our notation compatible to \cite{KV}. 

For $m\in\N_{0}$, recall the definition of $W(m)$, $N_{m}^{\star}$, $P_{f}^{\star}(t)$, and $\mu_{m}^{\star}$ in Section \ref{sec:prelim}. Also note that in this section the rank of $L_p$ is $d=3$ and $f(\bfx)$ is homogeneous of degree $n=2$.

 We want to express the number $N_{r,m}$ in Proposition \ref{pro:A_I} in terms of the integers $N^{\star}_{m}$.  Set 
 \begin{equation*}
     N_{r,m}^{\star}:=\left|\{\bfx\in W(r):\min(r,v_{p}(f(\bfx)))=m\}\right|.
 \end{equation*}
 Note that
\[
N_{r,m}^{\star} = (1-p^{-1}) p^r N_{r,m} = \begin{cases} \mu_m^{\star} p^{3r} & \text{if } m<r, \\ N_r^{\star} & \text{if } m=r. \end{cases}
\]
Let $z = p^{-(2s_1 + s_2 + s_3)}$, $w = p^{(s_1 + s_2 + s_3)}$. Similar to \cite[p.12]{KV}, we obtain:
\begin{align*}
A_{\{1\}}(\bss) =& \sum_{r=1}^{\infty} \sum_{m=0}^{\infty} N_{r,m} p^{-(2s_1 + s_2 + s_3) r} p^{(s_1 + s_2 + s_3) m} \\
=& \sum_{r=1}^{\infty} \sum_{m=0}^{\infty} \frac{1}{(1-p^{-1}) p^r}N_{r,m}^{\star} z^r w^m \\
=& \sum_{r=1}^{\infty} \frac{z^r}{(1-p^{-1}) p^r} \left( \sum_{m=0}^{r-1 } p^{3r} \left( \frac{N_m^{\star}}{p^{3m}} - \frac{N_{m+1}^{\star}}{p^{3(m+1)}} - \delta_{m,0} p^{-3} \right)w^m +  N_r^{\star}w^r \right) \\
=& \frac{(p^2+p+1) z }{1-p^2 z} + \frac{1-w^{-1}}{(1-p^{-1})(1-p^2 z)} (P_f^{\star}(p^2 z w) -1).
\end{align*}
As $p^2 z w = p^{-s_1+2}$, we substitute
\[
P_{f}^{\star}(p^2 z w) - 1 = \frac{ p^{-s_1+2} (1 - p^{-3} - (1-p^{-2s_1+1}) Z_f(s_1 - 2))}{1-p^{-s_1+2}}
\]
to show
\begin{align*}
A_{\{1\}} (\bss)=\frac{(1+p+p^2)p^{-s_1}}{1-p^{2-s_1}} - \frac{ p^{2-s_1} (1-p^{-s_1-s_2-s_3})(1-p^{1-2s_1})}{(1-p^{-1})(1-p^{2-s_1})(1-p^{2-2s_1-s_2-s_3})} Z_f(s_1-2).
\end{align*}

Together with Proposition \ref{pro:A_I} and write $t_1=p^{-s_1},t_2=p^{-s_2},t_3=p^{-s_3}$, we get
\begin{align*}
  A_{\emptyset}(\boldsymbol{s})&=1,\\
   A_{\{2\}}(\boldsymbol{s})&=\left(1+p^{-1}+p^{-2}\right)\left(\frac{p^{2}t_{1}t_{2}}{1-p^{2}t_{1}t_{2}}\right),\\
    A_{\{1\}}(\boldsymbol{s})&=\frac{(1+p+p^2)t_1}{1-p^{2}t_1} - \frac{ p^{2}t_1 (1-t_1t_2t_3)(1-pt_{1}^{2})}{(1-p^{-1})(1-p^{2}t_1)(1-p^{2}t_1^2t_2t_3)} Z_f(s_1-2),\\
    A_{\{1,2\}}(\boldsymbol{s})+A_{\{1\}}(\boldsymbol{s})&=A_{\{1\}}(\boldsymbol{s})\left(\frac{1+pt_{1}t_{2}}{1-p^{2}t_{1}t_{2}}\right),  
\end{align*}
giving
\begin{align*}
    \zeta_{L(\Zp)}^{\co}(\boldsymbol{s})&=\frac{1}{1-t_{1}t_{2}t_{3}}\sum_{I\subseteq\{1,2\}}A_{I}(\boldsymbol{s})\\
    &=\zeta_{\Z_{p}^3}^{\co}(\boldsymbol{s})-\frac{Z_{f}(s_{1}-2)p^{2}t_{1}(1+pt_{1}t_{2})(1-pt_{1}^{2})}{(1-p^{2}t_{1})(1-p^{2}t_{1}t_{2})(1-p^{2}t_{1}^2t_{2}t_{3})(1-p^{-1})}.
\end{align*}
Now, passing from $L$ to $p^{i}L$ amounts to replacing $f$ by $p^{i}f$. It is clear, however, that $Z_{p^{i}f}(s)=t^{i}Z_{f}(s)$, so that
\[Z_{p^{i}f}(s_{1}-2)=(p^2t_1)^{i}Z_{p^{i}f}(s_{1}-2).\]

Hence we get
\begin{align*}
    \zeta_{p^{i}L}^{\co}(\boldsymbol{s})=\zeta_{\Z_{p}^3}^{\co}(\boldsymbol{s})-\frac{Z_{f}(s_{1}-2)(p^{2}t_{1})^{i+1}(1+pt_{1}t_{2})(1-pt_{1}^{2})}{(1-p^{2}t_{1})(1-p^{2}t_{1}t_{2})(1-p^{2}t_{1}^2t_{2}t_{3})(1-p^{-1})},
\end{align*}
proving Theorem \ref{thm:cozeta.3dim}. 
\begin{rem}
   As mentioned in Section \ref{sec:intro}  if we put $s_1=s_2=s_3=s$, we get the expected formula
\begin{align*}
    \zeta_{p^{i}L}(s)=\zeta_{\Z_{p}^3}(s)-\frac{Z_{f}(s-2)(p^2t)^{i+1}}{(1-p^{2}t)(1-p^{2}t^2)(1-p^{-1})}
\end{align*}
as given in \cite[Theorem 1.1]{KV}.
\end{rem}

\section{Explicit computations of subalgebra cotype zeta functions}\label{sec:rank3.explicit}

In this section we show how one uses Theorem \ref{thm:cozeta.3dim} to compute subalgebra cotype zeta functions of various $\Zp$-Lie algebras of rank 3. 

\subsection{The Heisenberg Lie algebra $H$}

Let \[H=\langle e_1,e_2,e_3\mid [e_1,e_2]=e_3\rangle_{\Z}\]
be the Heisenberg $\Z$-Lie algebra. This gives $f(\bfx)=x_{3}^2,$ with (cf. \cite[Section 4.1]{KV}) 
\[Z_{f}(s)=\frac{1-p^{-1}}{1-p^{-1-2s}}\]
for all primes $p$.

Therefore, our formula gives
\begin{align*}
    \zeta_{H(\Zp)}^{\co}(\boldsymbol{s})&=\zeta_{\Z_{p}^3}^{\co}(\boldsymbol{s})-\frac{Z_{f}(s_{1}-2)p^{2}t_{1}(1+pt_{1}t_{2})(1-pt_{1}^{2})}{(1-p^{2}t_{1})(1-p^{2}t_{1}t_{2})(1-p^{2}t_{1}^2t_{2}t_{3})(1-p^{-1})}\\   
    &=\frac{N_{H}(p,t_1,t_2,t_3)}{D_{H}(p,t_1,t_2,t_3)},
\end{align*}
where
\begin{align*}
    N_{H}(X, Y_1,Y_2,Y_3)&=1+Y_1+XY_1+X^2Y_1^2+Y_1Y_2+XY_1Y_2+XY_1^2Y_2+X^2Y_1^2Y_2\\
    &-X^2Y_1^3Y_2Y_3-X^3Y_1^3Y_2Y_3-X^3Y_1^4Y_2Y_3
    -X^4Y_1^4Y_2Y_3-X^2Y_1^3Y_2^2Y_3\\
    &-X^3Y_1^4Y_2^2Y_3-X^4Y_1^4Y_2^2Y_3-X^4Y_1^5Y_2^2Y_3, \\
    D_{H}(X,Y_1,Y_2,Y_3)&=(1-X^{3}Y_{1}^{2})(1-X^{2}Y_{1}Y_{2})(1-X^{2}Y_{1}^2Y_{2}Y_{3})(1-Y_{1}Y_{2}Y_{3}).   
\end{align*}

\subsection{The `simple' Lie algebra $\textrm{sl}_2(\Z)$}
Let 
\[\textrm{sl}_2(\Z)=\langle e_1,e_2,e_3\mid [e_1,e_2]=e_3,[e_1,e_3]=-2e_1,[e_2,e_3]=2e_2\rangle_{\Z}\]
be a simple $\Z$-Lie algebra of rank 3. This gives $f(\bfx)=x_{3}^2+4x_1x_2,$ with (cf. \cite[Section 4.2]{KV})
\[Z_{f}(s-2)=\begin{cases}
    \frac{(1-p^{-1})(1-p^{-1-s})}{(1-p^{1-2s})(1-p^{1-s})}&p>2,\\
    1-2^{-1}+8\cdot 2^{-2s} \frac{(1-2^{-1})(1-2^{-1-s})}{(1-2^{1-2s})(1-2^{1-s})}&p=2.
\end{cases}\]

Therefore, for $p>2$ we have
\begin{align*}
    \zeta_{\textrm{sl}_2(\Zp)}^{\co}(\boldsymbol{s})&=\zeta_{\Z_{p}^3}^{\co}(\boldsymbol{s})-\frac{Z_{f}(s_{1}-2)p^{2}t_{1}(1+pt_{1}t_{2})(1-pt_{1}^{2})}{(1-p^{2}t_{1})(1-p^{2}t_{1}t_{2})(1-p^{2}t_{1}^2t_{2}t_{3})(1-p^{-1})}\\
    &=\frac{N_{\textrm{sl}_2,1}(p,t_1,t_2,t_3)}{D_{\textrm{sl}_2,1}(p,t_1,t_2,t_3)},
\end{align*}
where
\begin{align*}
    N_{\textrm{sl}_2 , 1}(X,Y_1,Y_2,Y_3)&=1+Y_1+Y_1Y_2(1+X)-X^2Y_1^{4}Y_2^{2}Y_3-XY_1^3Y_2Y_3(1+X+XY_2)\\
    D_{\textrm{sl}_2 , 1}(X,Y_1,Y_2,Y_3)&=(1-XY_1)(1-X^2Y_1Y_2)(1-X^2Y_1^2Y_2Y_3)(1-Y_1Y_2Y_3). 
\end{align*}
For $p=2$ we have
\begin{align*}
    \zeta_{\textrm{sl}_2(\Z_2)}^{\co}(\boldsymbol{s})&=\zeta_{\Z_{2}^{3}}^{\co}(\boldsymbol{s})-\frac{Z_{f}(s_{1}-2)2^{2}t_{1}(1+2t_{1}t_{2})(1-2t_{1}^{2})}{(1-2^{2}t_{1})(1-2^{2}t_{1}t_{2})(1-2^{2}t_{1}^2t_{2}t_{3})(1-2^{-1})}\\ 
    &=\frac{1+t_1+6t_1^2+3t_1t_2+12t_1^3t_2-12t_1^3t_2t_3-4t_1^3t_2^2t_3-16t_1^4t_2^2t_3}{(1-2t_1)(1-4t_1t_2)(1-4t_1^2t_2t_3)(1-t_1t_2t_3)}\\
    &=\frac{N_{\textrm{sl}_2,2}(2,t_1,t_2,t_3)}{D_{\textrm{sl}_2,2}(2,t_1,t_2,t_3)}.
\end{align*}

\subsection{Solvable but non-nilpotent Lie algebra $L_1$}

Let
\[L_1=\langle e_1,e_2,e_3\mid [e_1,e_2]=e_3,[e_1,e_3]=e_2\rangle_{\Z}\]
be a solvable but non-nilpotent $\Z$-Lie algebra of rank 3. This gives $f(\bfx)=x_3^2-x_2^2$. 

For $p>2$ one can check that  (cf. \cite[Section 4.4]{KV})
\[Z_{f}(s)=\left(\frac{1-p^{-1}}{1-p^{-1-s}}\right)^2,\]
giving 
\begin{align*}
    \zeta_{L_1(\Zp)}^{\co}(\boldsymbol{s})&=\zeta_{\Z_{p}^3}^{\co}(\boldsymbol{s})-\frac{Z_{f}(s_{1}-2)p^{2}t_{1}(1+pt_{1}t_{2})(1-pt_{1}^{2})}{(1-p^{2}t_{1})(1-p^{2}t_{1}t_{2})(1-p^{2}t_{1}^2t_{2}t_{3})(1-p^{-1})}\\ 
    &=\frac{N_{L_1,1}(p,t_1,t_2,t_3)}{D_{L_1,1}(p,t_1,t_2,t_3)},
\end{align*}
where
\begin{align*}
    N_{L_1,1}(X,Y_1,Y_2,Y_3)&=1+Y_2-2XY_1^2+Y_1Y_2+XY_1Y_2-XY_1^2Y_2-X^2Y_1^3Y_2\\
    &-XY_1^2Y_2Y_3-X^2Y_1^3Y_2Y_3+X^2Y_1^4Y_2Y_3+X^3Y_1^4Y_2Y_3\\&-2X^2Y_1^3Y_2^2Y_3+X^3Y_1^4Y_2^2Y_3+X^3Y_1^5Y_2^2Y_3\\
    D_{L_1,1}(X,Y_1,Y_2,Y_3)&=(1-XY_1)^2(1-X^2Y_1Y_2)(1-X^2Y_1^2Y_2Y_3)(1-Y_1Y_2Y_3).
\end{align*}

For $p=2$, write $(x_2,x_3)=(x+y,y)$. We have
\begin{align*}
Z_f(s) &= \int |x|^{s} |x+2y|^{s} dx dy = \int_{x \in 2\Z_2+1} |x|^{s} |x+2y|^{s} dx dy + \int_{x \in 2\Z_2} |x|^{s} |x+2y|^{s} dx dy \\
&= \frac{1}{2}  +  \frac{1}{2} \int_{z\in \Z_2} |2z|^{s}|2z+2y|^{s} dz dy \\
&= \frac{1}{2} + \frac{1}{2} 4^{-s} \left( \frac{1-2^{-1}}{1-2^{-1-s} }\right)^2 = \frac{1}{4} \frac{2-2^{1-s}+2^{-2s}}{(1-2^{-1-s})^2}.
\end{align*}
Hence
\[Z_{f}(s_1-2)=
    \frac{(1-2^{-1})(1-2^{2-s_1}+2^{3-2s_1})}{(1-2^{1-s_1})^2},\]
    giving
\begin{align*}
    \zeta_{L_1(\Z_{2})}^{\co}(\boldsymbol{s})&=\zeta_{\Z_{2}^3}^{\co}(\boldsymbol{s})-\frac{Z_{f}(s_{1}-2)2^{2}t_{1}(1+2t_{1}t_{2})(1-2t_{1}^{2})}{(1-2^{2}t_{1})(1-2^{2}t_{1}t_{2})(1-2^{2}t_{1}^2t_{2}t_{3})(1-2^{-1})}\\ 
    &=\frac{N_{L_1,2}(2,t_1,t_2,t_3)}{D_{L_1,2}(2,t_1,t_2,t_3)},
\end{align*}
where
\begin{align*}
    N_{L_1,2}(2,Y_1,Y_2,Y_3)&=1-t_1+4t_1^2+4t_1^3-16t_1^4+3t_1t_2-6t_1^2t_2+12t_1^3t_2+8t_1^4t_2-32t_1^5t_2\\
    &-12t_1^3t_2t_3+8t_1^4t_2t_3+16t_1^5t_2t_3-4t_1^3t_2^2t_3-8t_1^4t_2^2t_3+32t_1^6t_2^2t_3,\\
    D_{L_1,2}(2,Y_1,Y_2,Y_3)&=(1-2t_1)^2(1-4t_1t_2)(1-4t_1^2t_2t_3)(1-t_1t_2t_3).  
\end{align*}

\subsection{Another solvable but non-nilpotent Lie algebra $L_2$}\label{subsec:L2}
Let
\[L_2=\langle e_1,e_2,e_3\mid [e_1,e_2]=e_3,[e_1,e_3]=-e_2\rangle_{\Z}\]
be another solvable but non-nilpotent $\Z$-Lie algbera of rank 3. This gives $f(\bfx)=x_2^2+x_3^2$. 

For $p>2$, $f(\bfx)$ is reducible over $\Fp$ if $p \equiv 1 \bmod{4}$ and irreducible if $p \equiv 3 \bmod{4}$, giving
\[
Z_{f}(s)=
\begin{cases}
\left(\frac{1-p^{-1}}{1-p^{-1-s}}\right)^2 & p \equiv 1 \bmod{4}, \\
\frac{1-p^{-2}}{1-p^{-2-2s}} & p \equiv 3 \bmod{4}.
\end{cases}
\]

For $p=2$, note that
\[
v_2(x^2 + y^2) = \begin{cases} 0 & 2 \nmid x + y ,\\ 1 & 2 \nmid x,y ,\\ 2+v_2( (\frac{x}{2})^2 + (\frac{y}{2})^2) & 2 \mid x,y.
\end{cases}
\]
Hence we have
\begin{align*}
Z_f(s) &= \int_{x+y  \in 2\Z_2 +1}|x^2+y^2|^s dx dy + \int_{x,y \in 2\Z_2 + 1}|x^2+y^2|^s dx dy + \int_{x,y \in 2\Z_2} |x^2+y^2|^s dx dy \\
&= \frac{1}{2} + \frac{2^{-s}}{4} + \frac{4^{-s}}{4} Z_f(s),
\end{align*}
giving
\[
Z_f(s) = \frac{2^{-1}}{1-2^{-1-s}}.
\]
Combining all those cases, we get
\begin{align*}
    Z_{f}(s-2)=\begin{cases}
    \left(\frac{1-p^{-1}}{1-p^{1-s}}\right)^2&p \equiv 1 \bmod{4},\\
    \frac{1-p^{-2}}{1-p^{2-2s}} &p \equiv 3 \bmod{4}, \\
    \frac{2^{-1}}{1-p^{1-s}}&p=2
\end{cases}=\frac{1-p^{-1}}{1-p^{1-s}}\frac{1- \chi(p)p^{-1}}{1-\chi(p)p^{1-s}},
\end{align*}
where $\chi$ is the quadratic character mod $4$ given as
\[
\chi(p) = \begin{cases} 1 & p \equiv 1 \bmod{4}, \\ -1 & p \equiv -1 \bmod{4},\\ 0 & p =2. \end{cases}
\]

Hence, for $p\equiv 1\bmod 4$ we get
\begin{align*}
    \zeta_{L_2(\Zp)}^{\co}(\boldsymbol{s})&=\zeta_{\Z_{p}^3}^{\co}(\boldsymbol{s})-\frac{Z_{f}(s_{1}-2)p^{2}t_{1}(1+pt_{1}t_{2})(1-pt_{1}^{2})}{(1-p^{2}t_{1})(1-p^{2}t_{1}t_{2})(1-p^{2}t_{1}^2t_{2}t_{3})(1-p^{-1})}\\ 
     &=\frac{N_{L_2,1}(p,t_1,t_2,t_3)}{D_{L_2,1}(p,t_1,t_2,t_3)},
\end{align*}
where
\begin{align*}
    N_{L_2,1}(X,Y_1,Y_2,Y_3)&=1+Y_2-2XY_1^2+Y_1Y_2+XY_1Y_2-XY_1^2Y_2-X^2Y_1^3Y_2\\
    &-XY_1^2Y_2Y_3-X^2Y_1^3Y_2Y_3+X^2Y_1^4Y_2Y_3+X^3Y_1^4Y_2Y_3\\&-2X^2Y_1^3Y_2^2Y_3+X^3Y_1^4Y_2^2Y_3+X^3Y_1^5Y_2^2Y_3,\\
    D_{L_2,1}(X,Y_1,Y_2,Y_3)&=(1-XY_1)^2(1-X^2Y_1Y_2)(1-X^2Y_1^2Y_2Y_3)(1-Y_1Y_2Y_3),
\end{align*}
for  $p\equiv 3\bmod 4$ we get
\begin{align*}
    \zeta_{L_2(\Zp)}^{\co}(\boldsymbol{s})&=\zeta_{\Z_{p}^3}^{\co}(\boldsymbol{s})-\frac{Z_{f}(s_{1}-2)p^{2}t_{1}(1+pt_{1}t_{2})(1-pt_{1}^{2})}{(1-p^{2}t_{1})(1-p^{2}t_{1}t_{2})(1-p^{2}t_{1}^2t_{2}t_{3})(1-p^{-1})}\\ 
     &=\frac{N_{L_2,3}(p,t_1,t_2,t_3)}{D_{L_2,3}(p,t_1,t_2,t_3)},
\end{align*}
where
\begin{align*}
    N_{L_2,3}(X,Y_1,Y_2,Y_3)&=1+Y_1+Y_1Y_2+XY_1Y_2+XY_1^2Y_2-X^2Y_1^3Y_2+XY_1^2Y_2Y_3\\&-X^2Y_1^3Y_2Y_3-X^2Y_1^4Y_2Y_3-X^3Y_1^4Y_2Y_3-X^4Y_1^4Y_2^2Y_3-X^3Y_1^5Y_2^2Y_3,\\
    D_{L_2,3}(X,Y_1,Y_2,Y_3)&=(1-X^2Y_1^2)(1-X^2Y_1Y_2)(1-X^2Y_1^2Y_2Y_3)(1-Y_1Y_2Y_3),
\end{align*}
and finally, for $p=2$, we get
\begin{align*}
    \zeta_{L_2(\Z_{2})}^{\co}(\boldsymbol{s})&=\zeta_{\Z_{2}^3}^{\co}(\boldsymbol{s})-\frac{Z_{f}(s_{1}-2)2^{2}t_{1}(1+2t_{1}t_{2})(1-2t_{1}^{2})}{(1-2^{2}t_{1})(1-2^{2}t_{1}t_{2})(1-2^{2}t_{1}^2t_{2}t_{3})(1-2^{-1})}\\
    &=\frac{1+t_1-2t_1^2+3t_1t_2-4t_1^3t_2-4t_1^3t_2t_3-4t_1^3t_2^2t_3}{(1-2t_1)(1-4t_1t_2)(1-4t_1^2t_2t_3)(1-t_1t_2t_3)}\\
    &=:\frac{N_{L_2,2}(2,t_1,t_2,t_3)}{D_{L_2,2}(2,t_1,t_2,t_3)}.
\end{align*}

\subsection{Solvable Lie algebras in general}


In \cite{G-SK}, Gonz\'{a}lez-S\'{a}nchez and Klopsch classified all the solvable $\Zp$-Lie algebras of rank 3 for $p\geq3$. Using this result, in \cite[Section 4.4]{KV} the authors showed that  except $\Zp^3$, $H(\Zp)$, and $L_1(\Zp)$, computing $\zeta_{L(\Z_{p})}^{\co}(\bss)$ for a solvable $L$ reduces to the computation of the Igusa zeta function for the polynomial
\[f(\bfx)=x_2^2-dx_3^2,\]
where $d\in\Zp$.

We distinguish two cases. If $d=p^{k}\rho$, where $\rho\in\Zp^{*}$ is a non-square modulo $p$, we define
\[Z_{1,k}(s):=Z_{f}(s)\]
If $d=p^{k}u^2$, where $u\in\Zp^{*}$, we set
\[Z_{2,k}(s):=Z_f(s).\]

From \cite[Section 4.4]{KV} we obtain that
\begin{align*}
    Z_{1,0}(s)&=\frac{1-p^{-2}}{1-p^{-2}t^2},&Z_{2,0}(s)&=\left(\frac{1-p^{-1}}{1-p^{-1}t}\right)^2,\\
    Z_{1,1}(s)&=Z_{2,1}(s)=\frac{1-p^{-1}}{1-p^{-1}t},&
    Z_{i,k+2}(s)&=p^{-1}t^2Z_{i,k}(s)+1-p^{-1}
\end{align*}
for $i\in\{1,2\}$ and $k\geq2$. This yields
\begin{pro}
  Let $L$ be a solvable $\Z$-Lie algebra of rank 3 except $\Z^3$, $H$, and $L_1$. Then, for $p\geq3$, suitable $k,\iota\in\N_{0}$ and $i\in\{1,2\}$, we have
   \begin{align*}
    \zeta_{L(\Zp)}^{\co}(\bss)=\zeta_{\Z_{p}^3}^{\co}(\boldsymbol{s})-\frac{Z_{i,k}(s_{1}-2)(p^{2}t_{1})^{\iota+1}(1+pt_{1}t_{2})(1-pt_{1}^{2})}{(1-p^{2}t_{1})(1-p^{2}t_{1}t_{2})(1-p^{2}t_{1}^2t_{2}t_{3})(1-p^{-1})}.
\end{align*} 
\end{pro}
Again, for $s_1=s_2=s_3=s$ we get the expected result
\begin{align*}
    \zeta_{L}(s)=
\zeta_{\Z_{p}^3}(s)-\frac{Z_{i,k}(s-2)(p^2t)^{\iota+1}}{(1-p^{2}t)(1-p^{2}t^2)(1-p^{-1})},    
\end{align*} 
as stated in \cite[Proposition 4.1]{KV}.
\begin{rem}
    Note that this formula does not cover $p=2$, and one needs to compute $Z_{f}(s)$ separately (as we did for $L_2$ in Section \ref{subsec:L2}) to compute  $\zeta_{L(\Zp)}^{\co}(\bss)$ for all primes. 
\end{rem}
\subsection{Uniformity for $\zeta_{L}(s)$ and $\zeta_{L}^{\co}(\bss)$}
Recall the definition of the uniformity in Definition \ref{def:uniform.co}. Our result
\begin{align*}
    \zeta_{L}^{\co}(\boldsymbol{s})=\zeta_{\Z_{p}^3}^{\co}(\boldsymbol{s})-\frac{Z_{f}(s_{1}-2)p^{2}t_{1}(1+pt_{1}t_{2})(1-pt_{1}^{2})}{(1-p^{2}t_{1})(1-p^{2}t_{1}t_{2})(1-p^{2}t_{1}^2t_{2}t_{3})(1-p^{-1})}
\end{align*}
implies that for any $\Z$-Lie algebra $L$ of rank 3, the uniformity of both $\zeta_{L}^{\co}(\boldsymbol{s})$ and $\zeta_{L}(s)$ only depends on the behavior of the same $Z_{f}(s-2)$ when $p$ varies. Hence in this case $\zeta_{L}^{\co}(\boldsymbol{s})$ and $\zeta_{L}(s)$ always have the same uniformity. 

Furthermore, as discussed in Section \ref{sec:prelim}, $Z_{f}(s)$ is related to the number of solutions
\[N_{m}=\left|\{\bfx\in(\Z/p^{m}\Z)^{d}|f(\bfx)=0\}\right|.\]

Luckily, since $d=3$ in our case, for any $L$ our polynomial $f$ is always either 0 or a homogeneous quadratic form as we observed. In \cite[Proposition 2.1]{Lee}, the second author proved that for any homogeneous quadratic form $f$  the number $N_{m}$ is always PORC (Polynomial On Residue Classes). In other words, there exists $N\in\N$ and finitely many polynomials $W_{i}(X)\in\Q[X]$ for $0\leq i\leq N-1$ such that for each prime $p$, if $p\equiv i\bmod N$, then $N_{m}=W_{i}(p)$. In particular, this implies that  both $\zeta_{L}^{\co}(\boldsymbol{s})$ and $\zeta_{L}(s)$ are always either uniform  (eg. $L=\Z^3,H,\textrm{sl}_2(\Z),L_1$) or at least finitely uniform (eg. $L=L_2$). This proves Theorem \ref{thm:dim3.uniformity}.

\begin{rem}
    As we initially defined $\zeta_{L}^{\co}(\bss)$ for any $\mcO$-Lie algebra $L$, there is no need to limit oneself to $\mcO=\Z$ and one can more or less compute $\zeta_{L(\mcO_{\mfp})}^{\co}(\bss)$ for any ring of integers $\mcO$ of a number field $K$ using similar approach. However, in general it is very difficult to compute $\zeta_{L(\mcO_{\mfp})}^{\co}(\bss)$ explicitly for all non-zero prime $\mfp\in\Spec(\mcO)$, even in our cases. Since this is what we need in Section \ref{sec:explicit}, we decided to focus on $\Z$ and state the formulas as explicit as possible.
\end{rem}

\section{Density results and other asymptotics}\label{sec:explicit}
Let $L\in\{\Z^{3},H,\textrm{sl}_2(\Z),L_1,L_2\}$ be a $\Z$-Lie algebra of rank 3. Since we calculated $\zeta_{L(\Zp)}^{\co}(\bss)$ for all primes $p$ in Section \ref{sec:rank3.explicit}, the Euler product
\[\zeta_{L}^{\co}(\bss)=\prod_{p,\,prime}\zeta_{L(\Zp)}^{\co}(\bss)\]
allows us to explicitly describe $\zeta_{L}^{\co}(\bss)$. For $1\leq m\leq 3$, the explicit formula for $\zeta_{L}^{\co}(\bss)$ gives the Dirichlet series counting subalgebras with corank at most $m$ directly as (e.g, \cite[Section 3]{CKK})
\begin{align*}
\zeta_{L}^{(m)}(s) &= \lim_{s_{m+1}\rightarrow \infty} \cdots \lim_{s_d \rightarrow \infty} \zeta_{L}^{\co}(s,\ldots,s, s_{m+1}, \ldots,s_d)\\
&= \prod_{p} \left( \lim_{s_{m+1}\rightarrow \infty} \cdots \lim_{s_d \rightarrow \infty} \zeta_{L(\Zp)}^{\co}(s,\ldots,s, s_{m+1}, \ldots,s_d) \right).
\end{align*}

Let $N_{L}^{(m)}$ denote the number of sublattices $\Lambda$ of $L$ of index less than $X$ such that $\Lambda$ has corank at most $m$, and let $P_{L}^{(m)}$ denote the proportion of sublattices of $L$ of corank at most $m$. In this section, using formulas in Section \ref{sec:rank3.explicit}, we compute $\zeta_{L}^{\co}(\bss)$ for specified $\Z$-Lie algebras of rank 3 and compute $N_{L}^{(1)}$, $N_{L}^{(2)}$, $N_{L}^{(3)}$, $P_{L}^{(1)}$, and $P_{L}^{(2)}$ accordingly. This gives us Theorem \ref{thm:numeric}.

In the expression $t_i = p^{-s_i}$ we considered in expressing local factors, the limit $s_i \rightarrow \infty$ amounts to simply letting $t_i = 0$.
\subsection{The abelian case $\Z^3$}
We start from $L=\Z^{3}$. Since
\begin{align*}
    \zeta_{\Z_{p}^3}^{\co}(\boldsymbol{s})=\frac{1+t_1+pt_1+t_1t_2+pt_1t_2+pt_1^2t_2}{(1-p^2t_1)(1-p^2t_1t_2)(1-t_1t_2t_3)}
\end{align*}
for all primes $p$, we get
\begin{align*}
\zeta_{\Z^3}^{(1)}(s) &= \prod_p  \left( \lim_{s_{2}\rightarrow \infty}\lim_{s_{3}\rightarrow \infty}\zeta_{\Z_{p}^3}^{\co}(s,s_2,s_3) \right)= \zeta(s-2)\prod_{p}(1+(p+1)p^{-s}), \\
\zeta_{\Z^3}^{(2)}(s) &=\prod_p \left( \lim_{s_{3}\rightarrow \infty}\zeta_{\Z_{p}^3}^{\co}(s,s,s_3) \right) =  \zeta(s-1)\zeta(s-2) \zeta(s)\zeta(3s)^{-1}, \\
\zeta_{\Z^3}^{(3)}(s) &=\prod_p \zeta_{\Z_{p}^3}^{\co}(s,s,s) =\zeta(s)\zeta(s-1)\zeta(s-2).
\end{align*}

All three series have a simple pole at $s=3$, so we have
\[
N_{\Z^3}^{(m)}(X) \sim \left(\lim_{s \rightarrow 3} (s-3) \zeta_{\Z^3}^{(i)}(s)  \right) \frac{X^3}{3},
\]
giving
\begin{align*}
N_{\Z^3}^{(1)}(X) &\sim 
\frac{1}{3} \prod_{p} (1+p^{-2}+p^{-3}) X^3, \\
N_{\Z^3}^{(2)}(X) &\sim 
\frac{\zeta(2) \zeta(3)}{3\zeta(9)}X^3, \\
N_{\Z^3}^{(3)}(X) & \sim \frac{\zeta(2)\zeta(3)}{3} X^3.
\end{align*}
 Since $N_{\Z^{3}}^{(3)}$ gives the total number of sublattices of index less than $X$ in $\Z^{3}$, 
 one also gets
\begin{align*}
    P_{\Z^{3}}^{(1)}&\rightarrow 0.885,&P_{\Z^{3}}^{(2)}&\rightarrow 0.998
\end{align*}
as $X\rightarrow \infty$.

\subsection{The Heisenberg Lie algebra $H$}
Recall
\[\zeta_{H(\Zp)}^{\co}(\bss)=\frac{N_{H}(p,t_1,t_2,t_3)}{D_{H}(p,t_1,t_2,t_3)},\]
where
\begin{align*}
    N_H&=1+t_1+pt_1+p^2t_1^2+t_1t_2+pt_1t_2+pt_1^2t_2+p^2t_1^2t_2-p^2t_1^3t_2t_3-p^3t_1^3t_2t_3-p^3t_1^4t_2t_3\\
    &-p^4t_1^4t_2t_3-p^2t_1^3t_2^2t_3-p^3t_1^4t_2^2t_3-p^4t_1^4t_2^2t_3-p^4t_1^5t_2^2t_3, \\
    D_H&=(1-p^{3}t_{1}^{2})(1-p^{2}t_{1}t_{2})(1-p^{2}t_{1}^2t_{2}t_{3})(1-t_{1}t_{2}t_{3})
\end{align*}
for all prime $p$. This gives
\begin{align*}
\zeta_{H}^{(1)}(s) &= \prod_p  \left( \lim_{s_{2}\rightarrow \infty}\lim_{s_{3}\rightarrow \infty}\zeta_{H(\Zp)}^{\co}(s,s_2,s_3) \right)= \zeta(s-1)\zeta(2s-3) \prod_p (1+p^{-s}-p^{1-2s}-p^{3-3s}), \\
\zeta_{H}^{(2)}(s) &=\prod_p \left( \lim_{s_{3}\rightarrow \infty}\zeta_{H(\Zp)}^{\co}(s,s,s_3) \right)\\ & = \zeta(s-1)\zeta(2s-2)\zeta(2s-3) \cdot\prod_p (1+p^{-s}+p^{-2s}-p^{3-3s}-p^{3-4s}-p^{2-4s})  \\
\zeta_{H}^{(3)}(s) &=\prod_p \zeta_{H(\Zp)}^{\co}(s,s,s) =\zeta(s)\zeta(s-1)\zeta(2s-2)\zeta(2s-3)\zeta(3s-3)^{-1}.
\end{align*}

All three series have a double pole at $s=2$, so we have
\[
N_{H}^{(m)}(X) \sim \left(\lim_{s \rightarrow 2} (s-2)^2 \zeta_H^{(m)}(s)  \right) \frac{X^2 \log X}{2},
\]
and one can calculate
\begin{align*}
N_{H}^{(1)}(X) &\sim 
\frac{1}{4} \prod_{p} (1+p^{-2}-2p^{-3}) X^2 \log X, \\
N_{H}^{(2)}(X) &\sim 
\frac{1}{4}\zeta(2) \prod_p (1+p^{-2}-p^{-3}+p^{-4}-p^{-5}-p^{-6}) X^2 \log X, \\
N_{H}^{(3)}(X) & \sim 
\frac{\zeta(2)^2}{4\zeta(3)} X^2 \log X,
\end{align*}
which also gives
\begin{align*}
    P_{H}^{(1)}&\rightarrow 0.492,&P_{H}^{(2)}&\rightarrow 0.975
\end{align*}
as $X\rightarrow \infty$.

\subsection{The rest of examples in Section \ref{sec:rank3.explicit}}
From the formulas in Section \ref{sec:rank3.explicit}, analogous computations give us
\begin{align*}
     P_{\textrm{sl}_2(\Z)}^{(1)}&\rightarrow 0.488&P_{\textrm{sl}_2(\Z)}^{(2)}&\rightarrow 0.974\\   P_{L_1}^{(1)}&\rightarrow 0.492&P_{L_1}^{(2)}&\rightarrow 0.975\\ P_{L_2}^{(1)}&\rightarrow 0.482&P_{L_2}^{(2)}&\rightarrow 0.970
\end{align*}
as $X\rightarrow \infty$. This completes the proof of Theorem \ref{thm:numeric}.

\subsection{General density result}
We finally note that any subalgebra cotype zeta functions of $\Z$-Lie algebras of rank 3 exhibit similar patterns to one of our examples discussed.

Define two $\Z_p$ quadratic forms {\it similar} if those two forms are equivalent up to $\Z_p^{\times}$-multiplication and $\GL(\Z_p)$-action. Write $f \sim g$ if two forms are similar.

Define the {\it discriminant} of an integral  $\Z_p$-quadratic form  $f = \sum_{i,j=1}^{n} a_{ij} x_i x_j$ $(a_{ij}=a_{ji} \in \frac{1}{2} \Z)$ for $p \neq 2$ as $\mathrm{Disc}(f):=\det \left(a_{ij} \right)_{i,j=1, \cdots, n}$.

\begin{lem} \label{lem:quad_similar}
(\cite{Kitaoka}, Theorem 5.2.4) For $p \neq 2$, any $\Z_p$-quadratic form $f$ of rank $r$ with invertible discriminant is $\GL_r(\Z_p)$-equivalent to $x_1^2 + \cdots + x_{r-1}^2 + D x_r^2$, where $D = \mathrm{Disc}(f)$. 
\end{lem}

\begin{thm}
Let $L$ be a  $\Z$-Lie algebra of rank 3 whose corresponding quadratic form is $f=
f(\bfx) = \bfx^t \mathcal{A}_{L} \bfx
$.    \begin{enumerate}
        \item If $f$ is of rank $0$, then $ \zeta_{L}^{\co}(\bss) = \zeta_{\Z^3}^{\co}(\bss)$ and all $\zeta_{L}^{(m)}(\bss)$ for $m=1,2,3$ have a simple pole at $s=3$.
        \item If $f$ is of rank $1$, then $\zeta_{L(\Zp)}^{\co}(\bss)$ is the same as $\zeta_{H(\Zp)}^{\co}(\bss)$ for all but finitely many $p$, and all $\zeta_{L}^{(m)}(s)$ have a double pole at $s=2$.
        \item If $f$ is of rank $3$, then $\zeta_{L(\Zp)}^{\co}(\bss)$ is the same as $\zeta_{\textrm{sl}_2(\Zp)}^{\co}(\bss)$ for all but finitely many $p$, and all $\zeta_{L}^{(m)}(s)$ have a simple pole at $s=2$.
        \item If $f$ is of rank $2$ and is reducible in $\Z$, then $\zeta_{L(\Zp)}^{\co}(\bss)$ is the same as $\zeta_{L_1(\Zp)}^{\co}(\bss)$ for all but finitely many $p$, and all $\zeta_{L}^{(m)}(s)$ have a double pole at $s=2$.        
        \item If $f$ is of rank $2$ and is irreducible, then all $\zeta_{L}^{(m)}(s)$ have a simple pole at $s=2$.
    \end{enumerate}
\end{thm}

\begin{proof}
As shown in Theorem \ref{thm:cozeta.3dim}, for a prime $p$ the local subalgebra cotype zeta function $\zeta_{L(\Zp)}^{\co}(\bss)$ only depends on the Igusa zeta function $Z_{f_p}(s)$ of the image $f_p$ of $f$ in $\Z_p$. Thus if two $\Z$-Lie algebras $L, L'$ of rank 3 with corresponding quadratic forms $f,g$ satisfies $Z_{f_p}(s) = Z_{g_p}(s)$ for all but finitely many $p$, then $\zeta_{L(\Zp)}^{\co}(\bss)$ is the same as $\zeta_{L'(\Zp)}^{\co}(\bss)$ for all but finitely many $p$. Meanwhile, since two similar quadratic forms have the same Igusa zeta functions, it suffices to show that $f_p \sim g_p$ for all but finitely many $p$.\\

\noindent \textbf{Case 1. $f$ is of rank $0$}: We have $f=0$, and $\Z_p^3$ has the corresponding quadratic form $g=0$. Hence $f_p \sim g_p$ for all $p$ and we get $ \zeta_{L}^{\co}(\bss) = \zeta_{\Z^3}^{\co}(\bss)$.\\

\noindent \textbf{Case 2. $f$ is of rank $1$}: By a change of basis we may assume $f = ax^2$ for $a \neq 0$, and then $f_p \sim x^2$ for $p \nmid a$. Note that $H(\Zp)$ has the corresponding quadratic form $g=x^2$ for all $p$. Hence $\zeta_{L(\Zp)}^{\co}(\bss)$ is the same as $\zeta_{H(\Zp)}^{\co}(\bss)$ for all but finitely many $p$.\\

\noindent \textbf{Case 3. $f$ is of rank $3$}: Let $D = \mathrm{Disc}(f)$. If $p \nmid D$, then $Df_p$ has discriminant $D^4$ and  by Lemma \ref{lem:quad_similar} we get
\[
f_p \sim Df_p \sim x_1^2 + x_2^2 + D^4 x_3^2 \sim x_1^2 + x_2^2+x_3^2.
\]
The Lie algebra $\textrm{sl}_2(\Z)$ has the corresponding quadratic form $g = x_3^2+4x_1 x_2$, and one can  check that $g_p \sim x_1^2+x_2^2+x_3^2$ for $p \nmid \mathrm{Disc}(g)$. Hence $\zeta_{L(\Zp)}^{\co}(\bss)$ is the same as $\zeta_{\textrm{sl}_2(\Zp)}^{\co}(\bss)$ for all but finitely many $p$.\\

\noindent \textbf{Case 4. $f$ is of rank $2$}: We may assume that $f = ax_1^2 + bx_1 x_2 + cx_2^2$ by a change of basis. For $D = b^2-4ac$, $f$ is reducible if and only if $D$ is square in $\Z$.

If $p \nmid 2a$, then one may `complete the square' to show $f_p \sim x_1^2 - D x_2^2$ for $D = b^2-4ac$. If $D$ is square then $x_1^2 - D x_2^2 \sim x_1^2-x_2^2$ for $p \nmid 2aD$, and $L_{1}(\Z)$ has the corresponding quadratic form $g =x_1^2-x_2^2$. Hence $\zeta_{L(\Zp)}^{\co}(\bss)$ is the same as $\zeta_{L_1(\Zp)}^{\co}(\bss)$ for all but finitely many $p$.

Finally, if $D$ is non-square, the Igusa zeta function of $f_p\sim  x_1^2-Dx_2^2$ for $p \nmid D$ can be described as
\[
Z_{f_p}(s-2) = \frac{1-p^{-1}}{1-pt} \frac{1-\chi_D(p)p^{-1}}{1- \chi_D(p) pt},
\]
where $t=p^{-s}$ and $\chi_D(p) = \left( \frac{D}{p} \right)$, using the similar method to the case of $L_2$. The formula $\zeta_{L}^{(m)}(s)$ can be described as almost in the similar way as examples of $\zeta_{L_2}^{(m)}(s)$.
\end{proof}
In particular, one can observe:
\begin{cor}
Let $L$ be a $\Z$-Lie algebra of rank 3. For $1\leq m\leq 3$, $\zeta_{L}^{(m)}(s)$ always have the same value of the largest  pole with the same order.    
\end{cor}
\section{Further questions}

Instead of counting all sublattices or subalgebras of finite index with given cotype, one may count certain subfamilies of subalgebras with additional algebraic properties in $L$. One prominent example is to count (two-sided) ideals of $L$. 

\begin{dfn}
  We define the  \emph{ideal cotype zeta function of $L$} as the Dirichlet generating series
\begin{align*}
    \zeta_{L}^{\ideal,\,\co}(\bss)=\sum_{\Lambda\ideal L,|L/\Lambda|<\infty}\alpha_1(\Lambda)^{-s_1}\cdots\alpha_{d}(\Lambda)^{-s_{d}},
\end{align*}
enumerating  finite-index (two-sided) ideals of $L$ of given cotype. We also define the \emph{ideal zeta function of $L$} to be $\zeta_{L}^{\ideal}(s)=\zeta_{L}^{\ideal,\,\co}(s,\ldots,s)$.
\end{dfn}
Primary decomposition again yields the Euler product 
		\begin{equation*}
		\zeta_{L}^{\ideal,\,\co}(\bss)
		=\prod_{\mfp\in\Spec(\mcO)\setminus\{(0)\}}\zeta_{L(\Gri_{\mfp})}^{\ideal,\,\co}(\bss).
	\end{equation*}

Similar questions in this article  can also be asked for $\zeta_{L}^{\ideal,\,\co}(\bss)$. For instance, one can analogously extend the definition of the uniformity from that of $\zeta_{L}^{\co}(\bss)$ to $\zeta_{L}^{\ideal,\,\co}(\bss)$. In \cite{duS1}, du Sautoy constructed a $\Z$-Lie algebra $L_E$  of rank 9 whose zeta functions $\zeta_{L_E}(s)$ and $\zeta_{L_E}^{\ideal}(s)$ are known to be non-uniform, and the non-uniformity came from the behavior of the number of the $\Fp$-points of an elliptic curve $E:Y^2=X^3-X$. Since $\zeta_{L_E}(s)=\zeta_{L_E}^{\co}(s,\ldots,s)$ and $\zeta_{L_E}^{\ideal}(s)=\zeta_{L_E}^{\ideal,\,\co}(s,\ldots,s)$ it immediately follows that $\zeta_{L_E}^{\co}(\bss)$ and $\zeta_{L_E}^{\ideal}(s)$ are also non-uniform. How about the converse?
\begin{qun}
For a given $\mcO$-algebra $L$, is the (global) subalgebra cotype zeta function $\zeta_{L}^{\co}(\bss)$  uniform/finitely uniform/non-uniform if and only if $\zeta_{L}(s)$ is uniform/finitely uniform/non-uniform? If not, can we find a $\mcO$-Lie algebra $L$ such that  $\zeta_{L}(s)$ is uniform but $\zeta_{L}^{\co}(\bss)$ is non-uniform? Can we ask the same question for $\zeta_{L}^{\ideal,\,\co}(\bss)$ and  $\zeta_{L}^{\ideal}(s)$?
\end{qun}

It is also known that unlike subalgebra zeta functions, the functional equation does not always hold for ideal zeta functions (cf. \cite{Voll,VollIMRN/19}). There are examples of $\mcO$-Lie algebras whose local ideal zeta functions do not satisfy the local functional equations. In \cite{LV,VollIMRN/19}, the authors proved that  $\zeta_{L(\Gri_{\mfp})}^{\ideal}(s)$ satisfies the local functional equation \emph{if} it satisfies the so-called ``\emph{homogeneity}'' condition (please refer to \cite{LV,VollIMRN/19} for details).

\begin{qun}
    If $L$ satisfies the homogeneity condition, then does $\zeta_{L(\Gri_{\mfp})}^{\ideal,\,\co}(\bss)$ satisfy the local functional equation of the form
     \begin{align*}
    \left.\zeta_{L(\Gri_{\mfp})}^{\ideal,\,\co}(\bss)\right|_{q\rightarrow q^{-1}}=(-1)^{a}q^{b-\sum_{i=1}^{d}c_is_i}\zeta_{L(\Gri_{\mfp})}^{\ideal,\,\co}(\bss)
    \end{align*}
    for all but finitely many prime ideals $\mfp$ of $\mcO$? Can we identify $a,b,c_{i}\in\N$ from $L$?
\end{qun}

Many explicit computations of $\zeta_{L}(s)$ or $\zeta_{L}^{\ideal}(s)$ and their variants for various $R$-algebras are recorded in the literature (please see \cite{duSW,Lee1, LV, Rossmann2}, and references therein). However, at current stage we only have very few explicit computations of $\zeta_{L}^{\co}(\bss)$ or $\zeta_{L}^{\ideal,\,\co}(\bss)$. In a sequel, we are going to provide a general framework of computing ideal cotype zeta functions $\zeta_{L}^{\ideal,\,\co}(\bss)$ of $L$, together with some explicit example like an ideal cotype zeta functions of the Heisenberg algebra $H$. All the questions and further research discussed in this article would benefit from more explicit computations.


\end{document}